\documentclass{article}
\usepackage{amssymb}
\usepackage{amsmath}
\usepackage{tabularx}
\newtheorem{customtheorem}{Theorem}
\newtheorem{theorem}{Theorem}
\numberwithin{theorem}{section}

\newtheorem{corollary}[theorem]{Corollary}

\newtheorem{example}[theorem]{Example}

\newtheorem{lemma}[theorem]{Lemma}

\newtheorem{proposition}[theorem]{Proposition}
\newtheorem{remark}[theorem]{Remark}

\newenvironment{proof}[1][Proof]{\noindent\textbf{#1.} }{\ \rule{0.5em}{0.5em}}
\title{\mbox{Stability of binomials over finite fields}}

\author{Mohamed Ayad, Laboratoire de Math\'ematiques Pures et Appliqu\'ees,\\ Universit\'e du Littoral, F-62228 Calais. France\\ E-mail: ayadmohamed502@yahoo.com \and
	Boualem Benseba, Laboratoire ACC (Arithm\'etique, Codage et Combinatoire),\\   U. S. T. H. B. El Alia, Alger. Alg\'erie.\\E-mail: b.benseba@usthb.dz\and
	Mohamed Madi, Laboratoire ACC (Arithm\'etique, Codage et Combinatoire),\\  U. S. T. H. B. El Alia, Alger. Alg\'erie.\\
	E-mail: mo7amadi@outlook.com}
\begin{document}
	\maketitle

	\hfill{\scriptsize

		
		
		\textbf{Abstract} A polynomial $f(x)$ over a field $K$ is said to be stable if all its iterates are irreducible over $K$. L. Danielson and B. Fein have shown that over a large class of fields $K$, if $f(x)$ is an irreducible monic binomial, then it is stable over $K$. In this paper it is proved that this result no longer holds over finite fields.  Necessary and sufficient conditions are given in order that a given binomial is stable over $\mathbb{F}_q$. 
These conditions are used to construct  a table listing the stable binomials over  $\mathbb{F}_q$ of the form $f(x)=x^d-a$, $a\in\mathbb{F}_q\setminus\{0,1\}$,  for $q \leq 27$ and $d \leq 10$. The paper ends with a brief link with Mersenne primes.		
		
\section{Introduction}
		\bigskip
		Let $K$ be a field and $f(x)$ be a non constant polynomial with coefficients in $K$. Set $f_0(x)=x$, $f_1(x)=x$ and $f_n(x)=\Big(f_{n-1}\circ f\Big)(x)$ for $n\geq 2$. Following R. W. K. Odoni \cite{O1}, this polynomial is said to be stable over $K$ if $f_n(x)$ is irreducible over $K$ for all $n\geq 0$. The first example of such a polynomial appears in \cite{O1}, where it is proved that $f(x)=x^2-x+1$ is stable over $\mathbb{Q}$. In \cite{O2}, the same author shows that any iterate of an Eiseistein polynomial is itself Eiseinstein, thus $f(x)$ is stable. This gives a class of stable polynomials over fraction fields of factorial domains. 
		
		This result implies that, given an integral domain $A$, with fraction field $K$ and algebraically independent variables, $s_1,s_2,\ldots,s_n$, the generic polynomial of degree $n$,
		\begin{equation*}G(s,\ldots,s_n,x)=x^n-s_1x^{n-1}+\cdots+(-1)^ns_n\in A[s_1,\ldots,s_n][x]\end{equation*}
		is stable over $K(s_1,\ldots,s_n)$.
		
		Inspired by this result, the paper \cite{A} considers the stability of the generic polynomial of the integers of a number field. More precisely, let $K$ be a number field of degree $n$, $\{\omega_1,\ldots,\omega_n\}$ be an integral basis. Let $u_1,\ldots,u_n$ be independent variables over $\mathbb{Q}$ and let 
		\begin{equation*}
		F(u_1,\ldots,u_n,x)=\prod_{i=1}^n\Big(x-(u_1\omega_1^{\sigma_i}+\cdots+u_n\omega_n^{\sigma_i}\Big ),\end{equation*}
		where $\sigma_1,\ldots,\sigma_n$ are the distinct embeddings of $K$ in $\mathbb{C}$. Then under some arithmetical conditions, this polynomial is stable over $\mathbb{Q}(u_1,\ldots,u_n)$. At our best knowledge the stability of the polynomial $F$, in the general case, is still open. 
		
		The stability of quadratic polynomials over various fields was the subject of \cite{AM1} and \cite{AM2}.
		
		For the stability of trinomials over finite fields, we refer to \cite{AMS} 
		
		In \cite{DF}, the following surprising result is proved. Let $f(x)=x^n-b$ be an irreducible polynomial over $K$, then $f(x)$ is stable over $K$ in the following cases
		\begin{itemize}
			\item [(i)] $K=\mathbb{Q}$ and $b\in\mathbb{Z}$.
			\item [(ii)] $K=\mathbb{Q}(t)$ and $b\in\mathbb{Z}[t]$.
			\item [(iii)] $K=F(t)$ and $b\in F[t]$, where $F$ is an arbitrary algebraically closed field.
			\item [(iv)] $K=F(t)$, $b\in F(t)\setminus F$, $n\geq 3$ and $F$ an arbitrary field of characteristic $0$.
		\end{itemize}
	
	We will see in the last section of this paper that this result does not hold over finite fields.	
		We will give an algorithm which answer the stability or not of a given binomial over a finite field. 
		
		For fixed integer $d\geq 2$, define  inductively the polynomials $%
						(P_{n}(x))_{n\geq 1}$ by $P_{1}(x)=x$ and $%
						P_{n}(x)=[P_{n-1}(x)]^{d}+(-1)^{d-1}x$ for $n\geq 2.$ 
						
Let $a\in\mathbb{F}_q^\star$ and let $n_0$ and $m_0$ be indices with $n_0<m_0$ such that $P_{n_0}(a)=P_{m_0}(a)$. Suppose that $n_0$ first and $m_0$ next are choosen to be minimal for this propery, then $\{P_{n}(a),n\geq 1\}=\{P_1(a),\ldots,P_{m_0-}(a)\}$. This set play an important role in the stability of $f(x)=x^d-a$. In section 4, the following result is proved.

\begin{customtheorem}\label{theorem A}
							Let $d\geq 2$ be an integer such that $d\neq 0\pmod 4 $ and let $%
							f(x)=x^{d}-a\in \mathbb{F} _{q}[x].$ Let $n_{0}$, $m_{0}$, with $n_{0}<m_{0},$ be
							minimal integers such that $P_{n_{0}}(a)=P_{m_{0}}(a).$ Then $f(x)$ is
							stable over $\mathbb{F} _{q}$ if and only if for any prime number $l\mid d$
							and any $k\in \left\{ 1,...,m_{0}-1\right\} ,$ $P_{k}(a)\notin \mathbb{F}
							_{q}^{l}.$  Here $P_{k}(a)\notin \mathbb{F}
							_{q}^{l}.$ means in particular $P_k(a)\neq 0$.
						\end{customtheorem}
						
						This theorem is used for the construction of the table placed at the end of the paper. Obviously, the integer $m_0$ defined above satisfies the condition $m_0\leq q+1$. Theorem 4.4 shows that $m_0\leq (q-1)/\delta+2$, where $\delta=gcd(q-1,d)$. There are examples where this bound is reached. Suppose that $f(x)$ is not stable and let $r_0\geq 1$ be the smallest integer such that $f_{r_0-1}(x)$ is irreducible and $f_{r_0}(x)$ is reducible over $\mathbb{F}_q$. The preceding theorem implies $r_0\leq m_0-1\leq (q-1)/\delta+1$. Notice that this bound for $r_0$ depends on $q$.

{\bf Question} Does there exist an integer $N$ such that for any prime power $q$ and any $f(x)=x^d-a\in \mathbb{F}_q[x]$, if $f_1(x),\ldots, f_N(x)$ are irreducible over $\mathbb{F}_q$, then $f(x)$ is stable?

The example $f(x)=x^2-12\in\mathbb{F}_{19}[x]$ extracted from the table in section 4, shows that $f_1(x),\ldots,f_5(x)$ are irreducible while $f_6(x)$is reducible. This shows that if $N$ exists, then $N\geq 6$.
		
		\vspace{0.2cm}
		
		\textbf{Notation}
		The following notations will be used through the paper.
		
		$\mathbb{F}_q$: the finite field with $q$ elements.
		
		$lc(f)$: leading coefficient of $f(x)$.
		
		$f_n(x)$: $n$-th iterate of $f(x)$.
		
		$N_{F/K}(\alpha)$: norm of $\alpha$ relatvely to the extension $F/K$.
		
		$m\mid n$: $m$ divides $n$.
		
		\section{Preliminary results}
		\begin{lemma}
		
			Let $q$ be a power of a prime $p\neq 2$, $l$ be a prime number and $\delta$ be a positive integer such that 
			$l\mid q-1$ and $\delta\equiv 1\pmod l$. Let $\alpha \in \mathbb{F} _{q}^{\ast }.$ Then $\alpha $ is 
			an $l$-th power in $\mathbb{F} _{q}^{\ast }$ if and only if $\alpha $ is an $l$-th power in $
			\mathbb{F} _{q^{\delta }}^{\ast }.$
		\end{lemma}
		
		\begin{proof}
			Let $\xi $ be a generator of $\mathbb{F} _{q^{\delta }}^{\ast },$ then $\eta
			=\xi ^{(q^{\delta }-1)/(q-1)}$ is a generator of $\mathbb{F} _{q}^{\ast }.$
			Set $\alpha =\eta ^{u}$ where $u$ is a 
			nonnegative integer, then :
			\begin{equation*}
			\alpha =\xi ^{u(q^{\delta }-1)/(q-1)}=\xi ^{u(1+q+\cdot \cdot \cdot
				+q^{\delta -1})}:=\xi ^{v}
			\end{equation*}
			
			We have $v=u(1+q+\cdots +q^{\delta -1})\equiv u\delta\pmod l\equiv u\pmod l.$
			Hence $\alpha $ is an $l$-th power in $\mathbb{F} _{q}^{\ast }$ if and only if $
			\alpha $ is an $l$-th power in $\mathbb{F} _{q^{\delta }}^{\ast }.$
		\end{proof}
		
		\begin{lemma}
		Let $p$ be an odd prime number, $q=p^{m}$ with $m\geq 1.$ Let $k\geq 1$ be
			an integer and $\alpha \in \mathbb{F} _{q^{k}}^{\ast }.$ Then :%
			\begin{itemize}
				\item [(i)] \begin{equation*} \alpha \text{ is a square in }\mathbb{F} _{q^{k}}\iff
				\alpha ^{(q^{k}-1)/2}=1\end{equation*}%
				
				\item [(ii)] $$-1\text{ is a square in $\mathbb{F} _{q^{k}}$}\iff	\begin{cases}
				k &\text{is even or,}\\
				k &\text{is odd and $q\equiv 1\pmod 4$} 
				\end{cases}$$
			\end{itemize}
			\end{lemma}
		\begin{proof}
			Let $\xi $ be a generator of $\mathbb{F} _{q^{k}}^{\ast }$ and set $\alpha
			=\xi ^{u},$ $u\geq 0.$ Then :
			
			\textbf{(i)}%
			\begin{align*}
			\alpha ^{(q^{k}-1)/2}=1 & \iff \xi ^{u(q^{k}-1)/2}=1 \\ 
			&  \\ 
			& \iff u(q^{k}-1)/2\equiv 0\pmod q^{k}-1 \\ 
			&  \\ 
			& \iff u\equiv 0\pmod 2\\ 
			&  \\ 
			& \iff \alpha \text{ is a square in }\mathbb{F} _{q^{k}}%
			\end{align*}
			
			\textbf{(ii) \ }By\textbf{\ (i) }we have\textbf{\ :}%
			\begin{align*}
			-1\text{ is a square in }\mathbb{F} _{q^{k}} & \iff (-1)^{(q^{k}-1)/2}=1 \\ 
			&  \\ 
			& \iff (q^{k}-1)/2\equiv 0\pmod 2 \\ 
			&  \\ 
			& \iff q^{k}-1\equiv 0\pmod 4 \\ 
			&  \\ 
			& \iff q^{k}\equiv 1\pmod 4 \\ 
			&  \\ 
			& \iff k\text{ is even or }k\text{ is odd and }q\equiv 1\pmod 4 %
			\end{align*}
		\end{proof}
		\begin{lemma}
			Let $\delta \geq 1$ an integer, $K=\mathbb{F} _{q}$ and $F=\mathbb{F}
			_{q^{\delta }}.$
			
			\textbf{1.} The norm map $N_{F/K}:F\rightarrow K$ is surjective and it maps $%
			F^{\ast }$ onto $K^{\ast }.$
			
			\textbf{2.} Let $\xi $ a generator of $F^{\ast }$ and $\eta =\xi
			^{(q^{\delta }-1)/(q-1)}.$ Then $\eta $ is a generator of $K^{\ast }.$
			
			\textbf{3.} Let $l$ be a prime number. If $l\nmid q-1,$ then the morphism $%
			\Phi _{l}:F^{\ast }\rightarrow K^{\ast }$ such that $\Phi _{l}=x^{l}$ is
			one-to-one and onto. If $l\mid q-1,$ then $Ker(\Phi _{l})=\left\{ x\in F^{\ast
			},x^{l}=1\right\} $ and $(F^{\ast }/Ker(\Phi _{l}))\simeq (K^{\ast })^{l}.$
			Moreover, for any $a\in F^{\ast },$ $a\in (F^{\ast })^{l}$ if and only if $%
			N_{F/K}(a)\in (K^{\ast })^{l}.$
		\end{lemma}
		\begin{proof}
			\textbf{1.} \textit{See \cite{DF}[, Theorem 2.28]}. 
			
			\textbf{2.} Since $F^{\ast }$ is
			cyclic of order $q^{\delta }-1$ generated by $\xi ,$ it contains a unique
			subgroup of order $q-1$ generated by $\eta =\xi ^{(q^{\delta }-1)/(q-1)}$,
			namely $K^{\ast }.$ Moreover%
			\begin{equation*}
			\begin{array}{ll}
			\eta ^{k} & =\xi ^{k(q^{\delta }-1)/(q-1)}=(\xi ^{k})^{(q^{\delta
				}-1)/(q-1)}=(\xi ^{k})^{1+q+\cdot \cdot \cdot +q^{\delta -1}} \\ 
			&  \\ 
			& =(\xi ^{k})(\xi ^{k})^{q}\cdot \cdot \cdot (\xi ^{k})^{q^{\delta
					-1}}=N_{F/K}(\xi ^{k}).%
			\end{array}%
			\end{equation*}
			
			\textbf{3.} The statements about $\Phi _{l}$ are obvious. We prove the last
			statement of \textbf{3.} Set $a=\xi ^{k}$ for some $k\geq 0.$ Then 
			\begin{equation*}
			N_{F/K}(a)=[N_{F/K}(\xi )]^{k}=[\xi ^{(q^{\delta }-1)/(q-1)}]^{k}=\eta ^{k}.
			\end{equation*}
			
			Hence%
			\begin{equation*}
			a\in (F^{\ast })^{l}\Leftrightarrow k\equiv 0\pmod l \Leftrightarrow
			N_{F/K}(a)\in (K^{\ast })^{l}.
			\end{equation*}
		\end{proof}
		\begin{lemma}
			Let $K$ be a field, $\lambda\in K^\star$ and $u(x)=\lambda x$. Let $f(x)\in K[x]\setminus K$ and $g(x)=\big(u^{-1}\circ f\circ u\big)(x)$. Then, for any $n\geq 1$, $f_n(x)$ and $g_n(x)$ have the same number of irreducible factors over $K$.
		\end{lemma}
		\begin{proof} Obvious.
		\end{proof}
		
\section{Irreducibility of binomials}
						In this section, we study the irreducibility of binomials $f(x)=x^d-a$ with  $a\in\mathbb{F}_q^\star$.
						\begin{example}
							Let $q$ be a power of a prime $p$ and $f(x)=x^d-a$, where $a\in\mathbb{F}_q^\star$. Suppose that $p\mid d$, then $f(x)$ is reducible over $\mathbb{F}_q$
						\end{example}
						\begin{proof}
							Set $d=pm$ and $q=p^k$, where $m$ and $k$ are positive integers. Then \begin{equation*}
							f(x)=x^{pm}-a^q=(x^m)^p-(a^{p^{k-1}})^p=\Big(x^m-a^{p^{k-1}}\Big)\Big((x^m)^{p-1}+\cdots+(a^{p^{k-1}})^{p-1}\Big),
							\end{equation*}
							hence $f(x)$ is reducible over $\mathbb{F}_q$.
						\end{proof}
						\begin{lemma}
							Let $K$ be a field and $f(x),g(x)\in K[x]\backslash K.$ Let $\alpha$ be
							a root of $f(x)$ in an algebraic closure of $K.$ Then $f\circ g(x)$ is
							irreducible over $K$ if and only if $f(x)$ is irreducible over $K$ and $%
							g(x)-\alpha$ is irreducible over $K(\alpha).$
						\end{lemma}
						
						\begin{proof}
							See \cite{T}.
						\end{proof}
						
						\begin{lemma}
							Let $K$ be a field and $n\geq 2$ an integer. Let $a\in K^{\ast }.$ Assume
							that for all prime numbers $l\mid d,$ we have $a\notin K^{l}$ and if $4\mid
							d,$ then $a\notin -4K^{4}.$ Then $x^{d}-a$ is irreducible over $K.$ The
							converse is true.
						\end{lemma}
						
						\begin{proof} For the first part See \cite{L}[ Chapter 8, Theorem 16]. For the second part, let $l$ be a pirme divisor of $d$ and $a=b^{l},$ with $
							b\in K^{\ast }.$ Set $d=tl.$ Then
							\begin{align*}
							x^{d}-a & =x^{tl}-b^{l} \\ 
							&  \\ 
							& =(x^t)^{l}-b^{l} \\ 
							&  \\ 
							& =(x^t-b)u(x^t)\text{ for some }u(X)\in K[X] \\ 
							\end{align*}
							
							Thus $x^{d}-a$ is reducible. If $4\mid d,$ $d=4t$ and $a=-4b^{4}.$ Then :%
							\begin{equation*}
							x^{d}-a=x^{4t}+4b^{4}=(x^{2t}+2bx^{t}+2b^{2})(x^{2t}-2bx^{t}+2b^{2}),
							\end{equation*}
							hence the result.
						\end{proof}
						
						This result about the irreducibility of binomials is valid over any field. In \cite{LN}, the same problem of irreducibility is stated specifically over finite fields. Here is the content.
						\begin{lemma}
						Let $d\geq 2$ be an integer and $a\in \mathbb{F}_q^\star$. Then the binomial $x^d-a$ is irreducible in $ \mathbb{F}_q[x]$ if and only if the following two conditions are satisfied:
							\begin{itemize}
								\item [(i)] each prime factor of $d$ divides the order $e$ of $a$ in  $\mathbb{F}_q^\star$ but not $(q-1)/e$;
																\item [(ii)] $q\equiv 1\pmod 4$ if $d\equiv 0\pmod 4$.
							\end{itemize}
						\end{lemma}
						\begin{proof}
							See Theorem 3.75 in \cite{LN}.
						\end{proof}
						
						The conditions $(C)$ contained in Lemma 3.3 (resp. $(C^{'})$ contained in Lemma 3.4) are equivalent to the irreducibility of the binomial $x^d-a$, hence $(C)$ and $(C^{'})$ are equivalent. But at the first glance, it is not so obvious that they express the same meaning. So we prove the following.
												\begin{proposition}
Let $l$ be a prime number and $a\in\mathbb{F}_q^\star$. Let $e$ be the order of $a$ in $\mathbb{F}_q^\star$. Then the following propositions are equivalent:
							\begin{itemize}
								\item [(i)]	$l\mid e$ and 	$l\nmid (q-1)/e$;
								\item [(ii)] $a\not\in(\mathbb{F}_q^\star)^l$.
							\end{itemize}
						\end{proposition}
								\newpage
						\begin{proof} 
							\begin{itemize}
								\item[$\bullet$] $(i)\Rightarrow (ii).$ Suppose that $a\in(\mathbb{F}_q^\star)^l$ and let $\xi$ be a generator of  $\mathbb{F}_q^\star$. Then $a=\xi^{lu}$, where $u$ is a non negative integer. Moreover, we have $e=(q-1)/gcd(q-1,lu)$. Let $\delta=gcd(q-1,lu)$, then $\delta=(q-1)/e$. we show that $l\nmid e$ or $l\mid (q-1) /e$. Suppose that $l\mid e$, then $l\mid q-1$, thus $l\mid \delta$. Therefore $l\mid (q-1)/e$.
								\item [$\bullet$] $(ii)\Rightarrow (i).$ Suppose that $a\not\in(\mathbb{F}_q^\star)^l$ and let $\xi$ be a generator of  $\mathbb{F}_q^\star$. Then $a=\xi^{u}$, where $u$ is a non negative	integer such that $u\neq 0\pmod l$. Let $\delta=gcd(q-1,u)$. Since $u\neq 0\pmod l$, then $l\mid q-1$; otherwise any element of $\mathbb{F}_q^\star$ is an $l$-th power. This implies $l\nmid \delta$ and then $l\mid (q-1)/\delta$; Since $e=(q-1)/\delta$, then $l\mid e$. Since $\delta=(q-1)/e$, then $l\nmid(q-1)/e$. 
							\end{itemize}
						\end{proof}
												\begin{corollary}
Let $q$ be a a power of a prime.
							\begin{itemize}
								\item[1.] Let $d\geq 2$ be an integer such that $d\neq 0\pmod 4$ and let $a\in\mathbb{F}_q^\star$.  Suppose that $x^d-a$ is irreducible over $\mathbb{F}_q$. Then any prime factor of $d$ divides $q-1$. Moreover $d\mid q^d-1$.
								\item[2.] Let $d\geq 2$ be an integer such that $d\neq 0\pmod 4$. Let $a$ and $b\in\mathbb{F}_q^\star$ and let $e(a)$ and $e(b)$ be their respective orders. If $e(a)=e(b)$, then $x^d-a$ is irreducible over $\mathbb{F}_q$ if and only if the same holds for $x^d-b$.
								\item[3.] 	Let $a\in\mathbb{F}_q^\star$ and let $d\neq 0\pmod 4$. Suppose that $q\not\equiv -1 \pmod 4$, in the case $d$ even. Then $x^d-a$ is irreducible over $\mathbb{F}_q$ if and only if the same property holds for  $x^d+a$.
								\item[4.] Let $d$ and $e$ be positive integer such that $d\geq 2$, $d\neq 0\pmod 4$ and $gcd(d,e)\neq 1$. Let $a\in\mathbb{F}_q^\star$, then $x^d-a^e$ is reducible over $\mathbb{F}_q$ 
								\item[5.] Let $d$ and $e$ be positive integer such that $d\geq 2$, $d\neq 0\pmod 4$ and $gcd(d,e)=1$. Let $a\in\mathbb{F}_q^\star$, then $x^d-a$ is irreducible over $\mathbb{F}_q$ if and only if $x^d-a^ e$ satisfies the same property.
								\item[6.]  Let $d$ be positive integer such that $d\neq 0\pmod 4$. Let $\hat d$ be the squarefree part of $d$. Let $a\in\mathbb{F}_q^\star$ and $b\in\mathbb{F}_q^{\star \hat d}$ then $x^d-a$ is irreducible over $\mathbb{F}_q$ if and only if $x^d-ab$ satisfies the same property.
								\item[7.] If $d_1$ and $d_2$ have the same prime factors with $d_1$ and $d_2\neq 0\pmod 4$, then $x^{d_1}-a$ is irreducible over $\mathbb{F}_q$ if and only if $x^{d_2}-a$ is.
								\item[8.] Let $d_1$, $d_2$ be integers at least equal to $2$, $(gcd(d_1,d_2)=1$, both not $0$ modulo $4$. Let $f_1(x)=x^{d_1}-a_1$, $f_2(x)=x^{d_2}-a_2$ and $f(x)=x^{d_1d_1}-a_1^{d_2}a_2^{d_1}$, where $a_1$ and $a_2\in\mathbb{F}_q^\star$. Then $f(x)$ is irreducible over $\mathbb{F}_q$ if and only if $f_1(x)$ and $f_2(x)$ have the same property.
							\end{itemize}
						\end{corollary} 
						\begin{proof} \begin{itemize}
								
								\item[1.]	Suppose that some prime factor $l$ of $d$ does not divide $q-1$, then any element $a\in\mathbb{F}_q^\star$ is an $l$-th power, hence $x^d-a$ is reducible over $\mathbb{F}_q$ which is a contradiction. For the second part, let $\alpha$ be a root of $x^d-a$ and let $A=\{n\in\mathbb{Z}, \alpha^n\in\mathbb{F}_q\}$. Since $\alpha^{q^d-1}=1$, then $A\neq\emptyset$. Obviously, $A$ is an ideal of $\mathbb{Z}$. Let $\delta$ be generator of $A$, then $\delta\mid d$. On the other hand, $x^d-a$ is the minimal polynomial of $\alpha$ over $\mathbb{F}_q$, hence  $x^d-a$ divides $x^\delta-b$  for some $b\in \mathbb{F}_q^\star$. This implies $d\leq \delta$ and then $d=\delta$. Now since $q^d-1\in A$, then $d\mid q^d-1$. 
								\item[2.] Obvious from Lemma 3.4.
								\item[3.] 	For symmetry reason, we just prove the necessity of the condition. Suppose that $x^d-a$ is reducible over $\mathbb{F}_q$ and let $l$ be a prime number dividing $d$ such that $a=b^l$ with $b\in \mathbb{F}_q$. If $l=2$ and $2\mid q-1$, then $q\equiv 1\pmod 4$ and then, by Lemma 2.2, $-1$ is a square in $\mathbb{F}_q^\star$. This implies $-a$ is a square, hence $x^d+a$ is reducible. If $l=2$ and $2\nmid q-1$, then any element of $\mathbb{F}_q^\star$ is a square. In particular $-a$ is a square and then  $x^d+a$ is reducible. If $l\neq 2$, then $-a=-b^l=(-b)^l$ and we get the same conclusion as before.
								\item[4.] Let $l$ be a prime factor of $gcd(d,e)$. Set $d=ld_1$ and $e=le_1$. Then \begin{equation*}
								x^d-a^e=x^{ld_1}-a^{le_1}=(x^{d_1})^{l}-(a^{e_1})^{l}=\Big(x^{d_1}-a^{e_1}\Big)\Big((x^{d_1})^{l-1}+\ldots+(a^{e_1})^{l-1}\Big),
								\end{equation*}
								hence $x^d-a$ is reducible over $\mathbb{F}_q$.
								\item[5.] \begin{itemize}	
									\item [$\bullet$] Suppose that  $x^d-a$ is reducible over $\mathbb{F}_q$ and let $l$ be a prime number such that $l\mid d$ and $a=b^l$ with $b\in\mathbb{F}_q^\star$, then $a^e=(b^e)^l$, thus $x^d-a^ e$  is reducible over $\mathbb{F}_q$. 
									\item [$\bullet$] Conversely, suppose that $x^d-a^ e$  is reducible over $\mathbb{F}_q$ and let $l$ be a prime number such that $l\mid d$ and $a^e=b^l$ with $b\in\mathbb{F}_q^\star$. Let $u$ and $v\in \mathbb{Z}$  such that $ud+ve=1$ and let $\delta$ such that $d=l\delta$, then \begin{equation*}a=a^{ud+ve}=a^{ud}a^{ve}=\Big(a^{u\delta}\Big)^l\Big(b^{v}\Big)^l=\Big(a^{u\delta}b^{v}\Big)^l,\end{equation*}
									hence $x^d-a$ is reducible over $\mathbb{F}_q$.	
								\end{itemize}
								
								\item[6.] \begin{itemize}
									\item [$\bullet$] Suppose that $x^d-a$ is reducible over $\mathbb{F}_q$ and let $l$ be a prime number dividing $d$ such that $a$ is an $l$-power, then $ab$ is also an $l$-th power. This implies $x^d-ab$ is  reducible over $\mathbb{F}_q$.
									\item [$\bullet$] Suppose that $x^d-ab$ is reducible over $\mathbb{F}_q$ and let $l$ be a prime number dividing $d$ such that $ab$ is an $l$-th power. Multiplying by $b^{-1}$, which is an $l$-th power, shows that $a$ is an $l$-th power. Hence $x^d-a$ is reducible over $\mathbb{F}_q$.
								\end{itemize}
								\item[7.] Suppose that $x^{d_1}-a$ is reducible over $\mathbb{F}_q$ and let $l\mid d$ be a prime such that $a=b^l$ with $b\in\mathbb{F}_q^\star$. since $l\mid d_2$, then $x^{d_2}-a$ is reducible over $\mathbb{F}_q$.
								\item[8.] \begin{itemize}
									\item [$\bullet$] Sufficiency of the condition. Suppose, by contadiction that $f(x)$ is reducible over $\mathbb{F}_q$ and let $l$ be a prime number dividing $d_1d_2$ and $b\in\mathbb{F}_q^\star$ such that $a_1^{d_2}a_2^{d_1}=b^l$. We may suppose that $l\mid d_1$ and $l\nmid d_2$ since the proof is similar for the other case. Let $\xi$ be generator of $\mathbb{F}_q^\star$. Set $a_1=\xi^{u_1}$, $a_2=\xi^{u_2}$ and $b=\xi^{v}$, then $\xi^{u_1d_2+u_2d_1}=\xi^{vl}$, hence  $u_1d_2+u_2d_1\equiv vl\pmod{q-1}$. Let $w\in\mathbb{Z}$ such that $u_1d_2+u_2d_1=vl+w(q-1)$. If $l\nmid q-1$, then $f_1(x)$ is reducible by item 1. of the Corollary . Suppose next that $l\mid q-1$, then $l\mid u_1$, thus $a_1$ is an $l$-th power for a prime factor of $d_1$. This implies $f_1(x)$ is reducible.
									\item [$\bullet$] Necessity of the condition. Suppose that one of $f_1(x)$, $f_2(x$, say $f_1(x)$, is reducible over $\mathbb{F}_q$. Let $l$ be a prime number such that $l\mid d_1$ and $a_1=b^l$ with $b\in\mathbb{F}_q^\star$. Let $\delta_1\in\mathbb{Z}$ such that $d_1=l\delta_1$, then $a_1^{d_2}a_2^{d_1}=b^{ld_2}a_2^{l\delta_1}=\Big(b^{d_2}a_2^{\delta_1}\Big)^l$, hence $f(x)$ is reducible over  $\mathbb{F}_q$.
								\end{itemize}
							\end{itemize}
						\end{proof}
						
						Notice that the assumptions in $1.$ of Corollary 3.6 hold if $gcd(d,q-1)=1$.
						\section{Stability of binomials}
						Let $d\geq 2$ be an integer. Define inductively the polynomials $%
						(P_{n}(x))_{n\geq 1}$ by $P_{1}(x)=x$ and $%
						P_{n}(x)=[P_{n-1}(x)]^{d}+(-1)^{d-1}x$ for $n\geq 2.$ These polynomials will
						be used in the sequel.
						
						\begin{lemma}
Let $d\geq 2$ be an integer such that $d\neq 0\pmod 4 $ and let $%
							f(x)=x^{d}-a,$ where $a\in \mathbb{F} _{q}{}^{\ast }.$
							
							\textbf{1.} Suppose that $f_{n-1}(x)$ is irreducible and $f_{n}(x)$ is
							reducible over $\mathbb{F} _{q}$ for some $n\geq 2.$ Then there exists a
							prime number $l\mid d,$ such that $P_{n}(a)$ is an $l$-$th$ power in $%
							\mathbb{F} _{q}^\star.$
							
							\textbf{2.} If $P_n(a)=0$ for some $n\geq 2$, then $f(x)$ is reducible over $\mathbb{F} _{q}$.
							If for some $n\geq 1,$ $P_{n}(a)$ is an $l$-$th$
							power in $\mathbb{F} _{q}^\star$, for some prime divisor $l$ of $d$, then $f_{n}(x)$
							is reducible over $\mathbb{F} _{q}.$
						\end{lemma}
						
						\begin{proof}
							\textbf{1.} Let $(\alpha _{n})_{n\geq 1}$ be elements of $\overline{%
								\mathbb{F} _{q}}$ such that $f(\alpha _{1})=0$ and $f(\alpha _{n})=\alpha
							_{n-1}$ for $n\geq 2.$ Set $\beta _{k}=\alpha _{k}+P_{n-k}(a).$ We have $%
							f(\alpha _{n})=\alpha _{n}^{d}-a=\alpha _{n-1},$ hence $\alpha _{n}$ is a
							root of $x^{d}-a-\alpha _{n-1}.$ By Lemma 3.2 this polynomial is
							reducible over $\mathbb{F} _{q}(\alpha _{n-1}),$ hence by Lemma 3.3,
							there exists a prime number $l\mid d$ such that $a+\alpha _{n-1}=b_{1}^{l}$
							for some $b_{1}\in \mathbb{F} _{q}(\alpha _{n-1}),$ thus $P_{1}(a)+\alpha
							_{n-1}=b_{1}^{l},$ $b_{1}\in \mathbb{F} _{q}(\alpha _{n-1}).$ We prove by
							induction on $k\in \left\{ 1,...,n-1\right\} $ that \begin{equation}P_{k}(a)+\alpha
							_{n-k}=b_{k}^{l},\,\, \mbox{with}\, b_{k}\in \mathbb{F} _{q}(\alpha _{n-k})\end{equation}. Let $k\in \left\{ 1,...,n-2\right\} $ and suppose that $(1)$
							holds. Denote by $N$ the norm map $N_{\mathbb{F} _{q}(\alpha
								_{n-k})/\mathbb{F} _{q}(\alpha _{n-(k+1)})}.$ Then%
							\begin{align*}
							N(P_{k}(a)+\alpha _{n-k}) & =N(\beta _{n-k}) \\ 
							&  \\ 
							& =[N(b_{k})]^{l} \\ 
							&  \\ 
							& :=b_{k+1}^{l},%
							\end{align*}
							with $b_{k+1}\in \mathbb{F} _{q}(\alpha _{n-(k+1)}).$ Since $\alpha
							_{n-k}^{d}-a-\alpha _{n-(k+1)}=0,$ then \begin{equation}[\beta
							_{n-k}-P_{k}(a)]^{d}-a-\alpha _{n-(k+1)}=0,\end{equation} hence%
							\begin{align*}
							N(P_{k}(a)+\alpha _{n-k}) & =N(\beta _{n-k}) \\ 
							&  \\ 
							& =(-1)^{d}\left\{ [-P_{k}(a)]^{d}-a-\alpha _{n-(k+1)}\right\} \\ 
							&  \\ 
							& =[P_{k}(a)]^{d}+(-1)^{d-1}a+(-1)^{d-1}\alpha _{n-(k+1)} \\ 
							&  \\ 
							& =P_{k+1}(a)+(-1)^{d-1}\alpha _{n-(k+1)}.%
							\end{align*}
							
							If $d\equiv 1\pmod 2 $ we immediately obtain $P_{k+1}(a)+\alpha _{n-(k+1)}=b_{k+1}^{l}$ with $b_{k+1}\in \mathbb{F} _{q}(\alpha
							_{n-(k+1)}).$ If $d\equiv 0(\pmod 2 ,$ since $\alpha _{n-(k+1)}$ and $%
							-\alpha _{n-(k+1)}$ are conjugate over $\mathbb{F} _{q}(\alpha _{n-(k+2)})$
							then $P_{k+1}(a)+\alpha _{n-(k+1)}=\overline{b_{k+1}}^{l}$ where $%
							\overline{b_{k+1}}$ is a conjugate of $b_{k+1}$ over $\mathbb{F} _{q}(\alpha
							_{n-(k+2)}),$ thus our claim is proved. Applying the result for $k=n-1$ we
							get $P_{n-1}(a)+\alpha _{1}=b_{n-1}^{l}$ with $b_{n-1}\in \mathbb{F}
							_{q}(\alpha _{1}).$ This implies%
							\begin{equation*}
							\begin{array}{ll}
							N_{\mathbb{F} _{q}(\alpha _{1})/\mathbb{F} _{q}}(P_{n-1}(a)+\alpha _{1}) & 
							=N_{\mathbb{F} _{q}(\alpha _{1})/\mathbb{F} _{q}}(\beta _{1}) \\ 
							&  \\ 
							& =[N_{\mathbb{F} _{q}(\alpha _{1})/\mathbb{F} _{q}}(b_{n-1})]^{l} \\ 
							&  \\ 
							& =b_{n}^{l}\text{ with }b_{n}\in \mathbb{F} _{q}.%
							\end{array}%
							\end{equation*}
							
							Since $\alpha _{1}^{d}-a=0$ then $[\beta _{1}-P_{n-1}(a)]^{d}-a=0.$ Hence%
							\begin{equation*}
							\begin{array}{ll}
							N_{\mathbb{F} _{q}(\alpha _{1})/\mathbb{F} _{q}}(P_{n-1}(a)+\alpha _{1}) & 
							=N_{\mathbb{F} _{q}(\alpha _{1})/\mathbb{F} _{q}}(\beta _{1}) \\ 
							&  \\ 
							& =(-1)^{d}\left\{ [-P_{n-1}(a)]^{d}-a\right\} \\ 
							&  \\ 
							& =[P_{n-1}(a)]^{d}+(-1)^{d-1}a \\ 
							&  \\ 
							& =P_{n}(a)%
							\end{array}%
							\end{equation*}
							
							Thus $P_{n}(a)=b_{n}^{l}$ where $b_{n}\in \mathbb{F} _{q}.$
							
							\textbf{2.} Suppose that $P_n(a) =0$ for some $n\geq 2$, then $\Big(P_{n-1}(a)\Big)^d+(-1)^{d-1}a=0$, hence $P_{n-1}(a)$ is a root of $x^d+(-1)^{d-1}a$ in $\mathbb{F}_q$. By Lemma 3.3 and item 3. of Corollary 3.6, we get $x^d-a$ is reducible over $\mathbb{F} _{q}.$
							
							Suppose that for some $n\geq 1,$ $P_{n}(a)$ is an $l$-$th$
							power in $\mathbb{F} _{q}^\star$ for some prime divisor $l$ of $d.$
							If $n=1,$ then according to Lemma 3.3, $f(x)$ is
							reducible over $\mathbb{F} _{q}.$ Suppose that $n\geq 2.$ We may suppose
							that $f_{1}(x),...,f_{n-1}(x)$ are irreducible over $\mathbb{F} _{q}.$ We
							have $P_{n}(a)=N_{\mathbb{F} _{q}(\alpha _{1})/\mathbb{F}
								_{q}}(P_{n-1}(a)+\alpha _{1}).$ Hence by Lemma 2.3, $%
							P_{n-1}(a)+\alpha _{1}$ is an $l$-$th$ power in $\mathbb{F} _{q}(\alpha
							_{1}).$ Suppose by induction that $P_{n-k}(a)+\alpha _{k}$ is an $l$-$th$
							power in $\mathbb{F} _{q}(\alpha _{k}).$ Since $P_{n-k}(a)+\alpha
							_{k}=N_{\mathbb{F} _{q}(\alpha _{k+1})/\mathbb{F} _{q}(\alpha
								_{k})}(P_{n-(k+1)}(a)+\alpha _{k+1}),$ then by Lemma 2.3 again
							\textbf{, }$P_{n-(k+1)}(a)+\alpha _{k+1}$ is an $l$-$th$ power in $%
							\mathbb{F} _{q}(\alpha _{k+1}).$ In particular for $k=n-1,$ we get that $%
							P_{1}(a)+\alpha _{n-1}$ which is $a+\alpha _{n-1}$ is an $l$-$th$ power in $%
							\mathbb{F} _{q}(\alpha _{n-1}).$ Since $\alpha _{n}^{d}-a-\alpha _{n-1}=0$
							then by Lemma 3.3, the polynomial $\alpha _{n}^{d}-(a+\alpha _{n-1})$
							is reducible over $\mathbb{F} _{q}(\alpha _{n-1}).$ Thus $f_{n}(x)$\textbf{\ 
							}is reducible over $\mathbb{F} _{q}$ by Lemma 2.2.
						\end{proof}
						
						For fixed $a\in \mathbb{F} _{q}^{\ast },$ the family $(P_{n}(a))_{n\geq 1}$
						is finite. More precisely, let $n_{0}$ and $m_{0}$ be indices such that $%
						n_{0}<m_{0}$ and $P_{n_{0}}(a)=P_{m_{0}}(a).$ Suppose that $n_{0}$ first and 
						$m_{0}$ next are chosen to be minimal for this property, then \begin{equation}\left\{
						P_{n}(a),n\geq 1\right\} =\left\{ P_{1}(a),...,P_{m_{0}-1}(a)\right\} .\end{equation}
						Moreover, we have $m_{0}\leq q+1.$
						
						\begin{theorem} Let $d\geq 2$ be an integer such that $d\neq 0\pmod 4 $ and let $%
							f(x)=x^{d}-a\in \mathbb{F} _{q}[x].$ Let $n_{0}$, $m_{0}$, with $n_{0}<m_{0},$ be
							the minimal integers such that $P_{n_{0}}(a)=P_{m_{0}}(a).$ Then $f(x)$ is
							stable over $\mathbb{F} _{q}$ if and only if for any prime number $l\mid d$
							and any $k\in \left\{ 1,...,m_{0}-1\right\} ,$ $P_{k}(a)\notin \mathbb{F}
							_{q}^{l}.$  Here $P_{k}(a)\notin \mathbb{F}
							_{q}^{l}.$ means in particular $P_k(a)\neq 0$.
						\end{theorem}
						
						\begin{proof}
							If for some $l\mid d$ and some $k\in \left\{ 1,...,m_{0}-1\right\}$, $P_k(a)$ is an $l$-th power, trivial or not, then by the preceding lemma, $f(x)$ is not stable. Conversely suppose that $f(x)$ is not stable over $
							\mathbb{F} _{q}$ and let $n$ be the smallest positive integer such that $%
							f_{n}(x)$\textbf{\ }is reducible over $\mathbb{F} _{q}.$ If $n=1,$ that is $%
							f(x)$ is reducible, then $a=P_{1}(a)$ is an $l$-$th$ power in $\mathbb{F}
							_{q}^\star$ for some prime divisor $l$ of $d$. Suppose
							that $n\geq 2,$ then by Lemma 4.1, there exists a prime number $%
							l\mid d$ such that $P_{n}(a)$ is a\bigskip n $l$-$th$ power in $\mathbb{F}
							_{q}^\star$. Since $P_{n}(a)=P_{k}(a)$ for some $k\in \left\{
							1,...,m_{0}-1\right\} ,$ then $P_{k}(a)$ is a\bigskip n $l$-$th$ power in $
							\mathbb{F} _{q}^\star$ as desired.
						\end{proof}
						\begin{theorem} Let $d\geq 2$ be an integer such that $d\neq 0\pmod 4 $ and let $%
							f(x)=x^{d}-a\in \mathbb{F} _{q}[x]$ with $a\neq 0$. Suppose that $P_n(a)$ is an $l$-th power for some postive integer $n$ and some prime number $l\mid d$. Let $n_0$ be the smallest positive integer satisfying this property. If $P_{n_0}(a)=0$ or $n_0=1$ then $f(x)$ is reducible over  $\mathbb{F}
							_{q}$. If $n_0\geq 2$, then $f_{n_0}$ is reducible while $f_{n_0-1}$ is irreducible over $\mathbb{F} _{q}$. In any case $x^d -P_{n_0}(a)$ is reducible over $\mathbb{F} _{q}$.
						\end{theorem}
						\begin{proof} If $P_{n_0}(a)=0$ or $n_0=1$ then, by the preceding lemma, $f(x)$ is reducible over  $\mathbb{F}_q$. Suppose that $n_0\geq 2$, the preceding lemma shows that $f_{n_0}(x)$  is reducible over  $\mathbb{F}_q$. Since $f(x)$ is irreducible over  $\mathbb{F}_q$ (otherwise $n_0=1$), then we may consider the greatest integer $m\in\{1,\ldots,n_0-1\}$ such that $f_m(x)$ is irreducible over  $\mathbb{F}_q$, which in turn implies $f_{m+1}(x)$ is reducible over  $\mathbb{F}_q$. Lemma 4.1 shows that $P_{m+1}(a)$ is an $l$-th power in $\mathbb{F}_q$. Since $m+1\leq n_0$ and $n_0$ is minimal, then $m+1=n_0$, hence $f_{n_0-1}(x)$ is irreducible over $\mathbb{F}_q$. The last statement is obvious and its proof will be omitted. 
						\end{proof}	

\begin{corollary} Let $f(x)=x^d-a\in\mathbb{F}_q[x]$ with $a\neq 0$ and $d\not\equiv 0\pmod 4$. Let $\delta$ be a positive integer such that $\delta\equiv 1\pmod l$ for any prime factor $l$ of $d$. Then $f(x)$ is stable over $\mathbb{F}_q$ if and only if it is stable over $ \mathbb{F}_{q^\delta}$.
\end{corollary}
\begin{proof} {\bf Sufficiency of the condition.} Obvious.

 {\bf Necessity of the condition.} By contradiction, suppose that $f(x)$ is not stable over $\mathbb{F}_{q^\delta}$. By Theorem 4.2, there exist an index $n$ and a prime number $l\mid d$ such that $P_n(a)=0$ or $P_n(a)\in \mathbb{F}_{q^\delta}^l$. In the first case $f(x)$ is not stable over $ \mathbb{F}_{q}$, a contradiction. Now we consider the second possibility. Since $f(x)$ is stable over $\mathbb{F}_q$, then in particular $f(x)$ is irreducible, hence by item 1. of corollary 3.6, any prime factor of $d$ divides $q-1$, thus $l\mid q-1$. Now Lemma 2.1 implies $P_n(a)\in \mathbb{F}_{q}^l$ contradicting the stability of $f(x)$ over $\mathbb{F}_{q}$.
\end{proof}						
						
						\begin{theorem} Let $d\geq 2$ be an integer such that $d\neq 0\pmod 4 $ and let $%
							f(x)=x^{d}-a\in \mathbb{F} _{q}[x]$ with $a\neq 0$. Let $\delta=gcd(q-1,d)$. Suppose that $P_n(a)\neq 0$ for any positive integer $n$. If $f_1(x),\ldots,f_{(q-1)/\delta+1}(x)$ are irreducible over  $\mathbb{F}_q$, then $f(x)$ is stable.
						\end{theorem}
						\begin{proof} Let \begin{equation*}H= \mathbb{F}_q^{\star (q-1)/\delta}\,\,\mbox{and}\,\, A=\{P_1(a),\ldots,P_{\frac{(q-1)}{\delta}+1}\}. 
							\end{equation*}
							$H$ is a subgroup of $\mathbb{F}_q^\star$, so it index is equal to $(q-1)/\delta$. Since the multiset $A$ contains $\frac{(q-1)}{\delta}+1$ elements, there exist $i$ and $j$ with $i<j$ such that $P_j(a)\equiv P_i(a)\pmod H$, that is $P_j(a)=b^{(q-1/d)} P_i(a)$ with $b\in\mathbb{F}_q^\star$.  Moreover we may suppose that $i$ and $j$ are minimal for this property. This implies $P_j(a)^d=P_i(a)^d$ and then, by definition of the polynomials $P_k(x)$, $P_{j+1}(a)=P_{i+1}(a)$. From the definition of $n_0$ and $m_0$, we deduce that $i+1=n_0$ and $j+1=m_0$. Since $j\leq \frac{(q-1)}{\delta}+1$, then $m_0-1=j\leq \frac{(q-1)}{\delta}+1$. Now the stability follows from Theorem 4.3.
						\end{proof}
						\begin{remark} The preceding theorem shows that $m_0\leq (q-1)/\delta+2$. Here are two examples where the bound is reached.
							\begin{itemize}
								\item [$\bullet$] $q=7$, $d=2$, $a=-2$, $\delta=2$ and $f(x)=x^2+2$. We have $P_1(-2)=-2$, $P_2(-2)=-1$, $P_3(-2)=3$, $P_4(-2)=-3$ and  $P_5(-2)=-3$, hence $m_0=5=(7-1)/2+2$.
								\item [$\bullet$] $q=7$, $d=3$, $a=2$, $\delta=3$ and $f(x)=x^2-2$. We have	$P_1(2)=2$, $P_2(2)=3$, $P_3(2)=1$ and  $P_4(2)=3$, hence $m_0=4=(7-1)/3+2$.
							\end{itemize}
						\end{remark}
						
						\begin{corollary} Let $a\in\mathbb{F}_q^\star$. Suppose that $d$ is odd, then $x^d-a$ is stable over $\mathbb{F}_q$ if and only if the same property holds for  $x^d+a$ 
						\end{corollary}
						\begin{proof}
							By induction on $n$, we show that $P_n(-a)=-P_n(a)$. Then the conclusion follows immediately from the preceding theorem and from item 3. of Corollary 3.6. This identity is true by assumption. Suppose it is true for the step $n-1$. Then \begin{equation*}P_n(-a)=\Big(P_{n-1}(-a)\Big)^d+(-a)=-\Big(P_{n-1}(a)\Big)^d-a=-P_n(a).\end{equation*}	
						\end{proof}
						\begin{remark} One verifies easily that $$f_n(0)=\begin{cases}&\mbox{$-P_n(a)$ for $n\geq 1$ if $d$ is odd}\\
							&\mbox{$-P_1(a)$ if $n=1$ and $P_n(a)$ for $n\geq 2$ if $d$ is even}.
							\end{cases}$$
						\end{remark}
						\vspace{0.5cm}
						Here we consider the case where the binomials are not monic.
						\begin{proposition} Let $g(x)=bx^d-c\in\mathbb{F}_q[x]$, with $b$ and $c\neq 0$. If some prime factor $l$ of $d$ does not divide $q-1$, then $g(x)$ is reducible over  $\mathbb{F}_q$. Suppose that any prime factor of $d$ divides $q-1$. Then there exist $a$ and $\lambda\in\mathbb{F}_q^\star$ such that $g(x)=\big(u^{-1}\circ f\circ u\big)(x)$, where $u(x)=\lambda x$ and $f(x)=x^d-a$. Moreover $g(x)$ is stable over $\mathbb{F}_q$ if and only if $f(x)$ is.
						\end{proposition}
						\begin{proof} If $l\mid d$ and $l\nmid q-1$, then $cb^{-1}\in \mathbb{F}_q^{\star l}$, hence $x^d-cb^{-1}$ is reducible over $\mathbb{F}_q$ and then the same holds for $g(x)$. Suppose that any prime factor of $d$ divides $q-1$. The condition:  there exist $a$ and $\lambda\in\mathbb{F}_q^\star$ such that $g(x)=\big(u^{-1}\circ f\circ u\big)(x)$ is equivalent to $b=\lambda^{d-1}$ and $c=a\lambda^{-1}$. Since $gcd(d-1,q-1)=1$, then the first equation determines $\lambda$. The second equation determines $a$. The statement about stability follows from Lemma 2.4 
							
						\end{proof}
						\begin{example} \textbf{1.} $f(x)=x^{2}+1\in \mathbb{F} _{3}[x].$ So $a=-1$ and $d=2.$
							
							$P_{1}(a)=P_{2}(a)=-1.$
							
							$-1$ is not a square in $\mathbb{F} _{3},$ hence $x^{2}+1$ is stable over $%
							\mathbb{F} _{3}.$
							
							$%
							\begin{array}{l}
							\end{array}%
							$
							
							\textbf{2. }$f(x)=x^{2}-2\in \mathbb{F} _{5}[x].$ So $a=2$ and $d=2.$
							
							$P_{1}(a)=P_{2}(a)=2$\textbf{\ }which is not a square in $\mathbb{F} _{5},$
							hence the stability of $x^{2}-2$ over $\mathbb{F} _{5}.$
							
							$%
							\begin{array}{l}
							\end{array}%
							$
							
							\textbf{3.} $f(x)=x^{2}+2\in \mathbb{F} _{5}[x].$ So $a=-2$ and $d=2.$
							
							$P_{1}(a)=-2$ and $P_{2}(a)=1$ which is a square in $\mathbb{F} _{5}.$ $%
							x^{2}+2$ is not stable over $\mathbb{F} _{5}$.
							
							$%
							\begin{array}{l}
							\end{array}%
							$
							
							\textbf{4.} We have $\mathbb{F} _{9}=\mathbb{F} _{3}(i)$ where $i^{2}=-1.$
							
							$%
							\begin{array}{l}
							\end{array}%
							$
							
							$\bullet f_{1}(x)=x^{2}-(1+i)\in \mathbb{F} _{9}[x].$ So $a=(1+i)$ and $d=2.$
							
							$P_{3}(a)=-1$ which is a square in $\mathbb{F} _{9},$ hence $x^{2}-(1+i)$ is
							not stable over $\mathbb{F} _{9}$.
							
							$%
							\begin{array}{l}
							\end{array}%
							$
							
							$\bullet f_{2}(x)=x^{2}+(1+i)\in \mathbb{F} _{9}[x].$ So $a=-(1+i)$ and $d=2.$
							
							$P_{2}(a)=1$ which is a square in $\mathbb{F} _{9},$ hence $x^{2}+(1+i)$ is
							not stable over $\mathbb{F} _{9}$.
							
							$%
							\begin{array}{l}
							\end{array}%
							$
							
							$\bullet f_{3}(x)=x^{2}-(1-i)\in \mathbb{F} _{9}[x].$ So $a=1-i$ and $d=2.$
							
							$P_{3}(a)=-1$ which is a square in $\mathbb{F} _{9},$ hence $x^{2}-(1-i)$ is
							not stable over $\mathbb{F} _{9}$.
							
							$%
							\begin{array}{l}
							\end{array}%
							$
							
							$\bullet f_{4}(x)=x^{2}-(i-1)\in \mathbb{F} _{9}[x].$ So $a=i-1$ and $d=2.$
							
							$P_{2}(a)=1$ which is a square in $\mathbb{F} _{9},$ hence $x^{2}-(i-1)$ is
							not stable over $\mathbb{F} _{9}$.
						\end{example}
							
						A Mersenne number is a positive integer of the form $2^m-1$, where $m$ is an integer at least equal to $2$. Set $q=2^m$. When $q-1$ is prime, this prime is called Mersenne prime. It is well known that if $q-1$ is a Mersenne prime then $m$ is prime. The converse is false, the first counterexample is $2^{11}-1=2047=23\times 89$. We prove the following.
						\begin{theorem} Let 
							\begin{itemize}
								\item [1.] Let $q$ be a non trivial prime power and $\alpha\in\mathbb{F}_q^\star$. Then $x^{q-1} -\alpha$ is irreducible over $\mathbb{F}_q$ if and only if $\alpha$ generates $\mathbb{F}_q^\star$.
								\item [2.]	Suppose that $q=2^m$, where $m\geq 2$ is an integer. Then the following conditions are equivalent.
								\begin{itemize}
									\item [(i)] For any $\alpha\in\mathbb{F}_q\setminus\{0,1\}$, $x^{q-1}-\alpha$ is stable over $\mathbb{F}_q$
									\item [(ii)] For any $\alpha\in\mathbb{F}_q\setminus\{0,1\}$, $x^{q-1}-\alpha$ is irreducible over $\mathbb{F}_q$
										\item [(iii)]	$q-1$ is a Mersenne prime.
									\end{itemize}
							\end{itemize}
						\end{theorem}
						\begin{proof}
							\begin{itemize}
								\item [1.] {\bf Necessity of the condition.} Let $e$ be the order of $\alpha$ in  $\mathbb{F}_q^\star$. Obviously $e\leq q-1$. On the other hand, let $l$ be a prime divisor of $q-1$. Suppose, by contradiction, that the $l$-adic valuations satisfy the condition $\nu_l(a) <\nu_l(q-1)$, then $l\mid (q-1) /e$, which contradicts the irreducibility of $x^{q-1}-\alpha$ (see Lemma 3.4). Therefore $\nu_l(a) =\nu_l(q-1)$, and then $e=q-1$.
								
								{\bf Sufficiency of the condition.} Since $e=q-1$, then, by Lemma 3.4 , $x^{q-1}-\alpha$ is irreducible over $\mathbb{F}_q$.  
								\item [2.]	
								\begin{itemize}
									\item[$\bullet$] $(i)\Rightarrow (ii).$ Obvious.
								\item[$\bullet$] $(ii)\Rightarrow (iii).$ By \textbf{1.}, for any $\alpha\in\mathbb{F}_q\setminus\{0,1\}$, $\alpha$ generates  $\mathbb{F}_q^\star$. This implies that the order of this cyclic group is a prime number, thus $q-1$ is prime.
								\item[$\bullet$] $(iii)\Rightarrow (i).$ By \textbf{1.}, for any $\alpha\in\mathbb{F}_q\setminus\{0,1\}$, $x^{q-1}-\alpha$ is irreducible over  $\mathbb{F}_q$. To get the stability, since  $(\mathbb{F}_q^{\star})^{q-1} =\{1\}$, we must show that $P_n(\alpha)\notin\{0,1\}$ for any $n\geq 1$. For $n=1$, this is proved above. Suppose, by induction, that it is true for $n\geq 1$. We have $P_{n+1}(\alpha)=P_{n}(\alpha)^{q-1}+(-1)^{q-2}\alpha=1+\alpha$. If $P_{n+1}(\alpha)=0$, then $\alpha=1$. If $P_{n+1}(\alpha)=1$, then $\alpha=0$. In both cases we reach a contradiction. Therefore $x^{q-1}-\alpha$ is stable over $\mathbb{F}_q$.
								\end{itemize}	
								\end{itemize}		
						\end{proof}
\newpage
In \cite{LN} a table of irreducible polynomials over $\mathbb{F}_q$, of degree $d$, for small $q$ and small $d$ is given. The following table lists the stable binomials $f(x)=x^d-a$ for $3\leq q\leq 27$, $2\leq d\leq 10$, $d\neq 0\pmod 4$ and $a\in\mathbb{F}_q\setminus\{0,1\}$. The values of $d$ for which there exists a prime number $l\mid d$ and $l\nmid q-1$ are omitted since in this case $f(x)$ is reducible over $\mathbb{F}_q$(see Corollary 3.6. The values of $d$ which are congruent to $0$ modulo $4$ are also omitted. For given $q$, $d$ and $a$, the table lists the sequence $[P_1(a)],\ldots,P_{m_0}(a)]$ (see the begining of section 4 for the definition of this sequence). One and only one of this list, say $P_n(a)$, is possibly bold. This means that $P_n(a)$ is an $l$-th power for some prime divisor $l$ of $d$ and $n$ is the smallest postive integer satisfying this property. This implies that $n$ is the smallest positive integer such that $f_n(x) $ is reducible over $\mathbb{F}_q$. If no element is bold then $f(x)$ is stable.
								The elements of $\mathbb{F}_q\setminus\{0,1\}$ are enumerated in the following way. If $q=p$ is a prime number, then $a=2,\ldots,p-1$. If $q=p^e$ with $e\geq 2$, then a generator $\alpha$ of $\mathbb{F}_q^\star$ is choosen and its minimal polynomial over $\mathbb{F}_p$, $M(x)$, is mentionned. In this case $a=\alpha,\ldots,\alpha^{q-2}$. If a binomial $f(x)$ of degree $d$ is revealed to be stable over $\mathbb{F}_q$ by this table, then we have an infinite list of polynomials having the same property. 
								\vspace{0.2cm}
								
								{\centerline{\bf{Table of stable polynomials} }}
								
	\begin{tabularx}{11.75cm}{|p{0.25cm}|X|p{0.25cm}|X|p{0.25cm}|X|}
	\hline
		\multicolumn{2}{|c|}{$q=3$} & \multicolumn{4}{c|}{$q=4,M(X)=X^{2}+X+1$}     
     \\
		\hline
		\multicolumn{2}{|c|}{$\underline{d=2}\hspace{3cm}$} & \multicolumn{2}{c|}{$\underline{d=3}\hspace{2cm}$}                             & \multicolumn{2}{c|}{$\underline{d=9}\hspace{3cm}$}  
	\\ 
	\multicolumn{2}{|c|}{} & \multicolumn{2}{c|}{}                             & \multicolumn{2}{c|}{} 
    \\ 
	 $a$ &  List $\qquad$ & $a$ &  List $\qquad$   &  $ a$   &   List $\qquad$        
	\\ 
	2  & $[2,2] s.$   &  $ \alpha  $ & $[\alpha,\alpha^2,\alpha^2] s.$    &   $ \alpha  $        & $[\alpha,\alpha^2,\alpha^2] s.$ 
	 \\
	&     &  $\alpha^2$ & $[\alpha^2,\alpha,\alpha] s.$                             & $\alpha^2$  & $[\alpha^2,\alpha,\alpha] s. $ 
	\\
    \hline
	\multicolumn{6}{|c|}{$q=5$}                                                                \\
	\hline
	\multicolumn{6}{|c|}{$\underline{d=2}$}                                   
	\\ 
	\multicolumn{6}{|c|}{} 
	 \\	
	2     &   $[2,2] s.$        & 3                         & $[3,\underline{1},3] ns.$     & 4          &  $[\underline{4},2,0,1] ns.$         7
	 \\
	\hline
	\multicolumn{6}{|c|}{$q=7$} 
	\\
	\hline
	\multicolumn{2}{|l|}{$\underline{d=2}\hspace{2cm}$} & \multicolumn{2}{l|}{$\underline{d=3}\hspace{2cm}$}                                & \multicolumn{2}{l|}{$\underline{d=6}\hspace{2cm}$}   
	\\
	2     &   $[\underline{2},2]s.$        & 2                         & $[2,3,\underline{1},3] s.$                        & 2          &  $[\underline{2},6,6] s.$          \\
	3          & $[3,6,5,\underline{1},5] s.$          & 3                         & $[3,2,4,4]s.$                        & 3 & $[3,5,5]s.$          
	 \\
	4          & $[\underline{4},5,0,3,5]ns.$          & 4                         & $[4,5,3,3]s.$                        & 4          & $[\underline{4},4] ns.$    
	\\
	5          & $[5,6,3,\underline{4},4] ns.$          & 5                         & $[5,4,\underline{6},4]ns.$    & 5          & $[5,3,3]s.$          
	 \\
	6          & $[6,\underline{2},5,2]ns.$          & 6                         & $[\underline{6},5,5]ns.$                        & 6          & $[\underline{6},2,2] ns.$ 
	\\
	\hline
	\end{tabularx}
			
\newpage
			
	\begin{tabularx}{11.75cm}{|p{0.25cm}|X|p{0.25cm}|X|p{0.25cm}|X|}
	\hline
		\multicolumn{2}{|c|}{$q=7$} & \multicolumn{2}{c|}{$q=8,$} & \multicolumn{2}{c|}{$q=9,$}   
	\\ 
	\multicolumn{2}{|c|}{} & \multicolumn{2}{c|}{{\footnotesize{}$M(X)=X^{3}+X+1$}{\footnotesize\par}} & \multicolumn{2}{c|}{{\footnotesize{}$M(X)=X^{2}+2X+2$}{\footnotesize\par}}   
	\\	
	\hline
	\multicolumn{2}{|c|}{$\underline{d=9}\hspace{3cm}$} & \multicolumn{2}{c|}{$\underline{d=7}\hspace{1.7cm}$}                                & \multicolumn{2}{c|}{$\underline{d=2}\hspace{2.8cm}$}   
	\\
		2      & $[2,3,\underline{1},3 ]ns.$             &                          $\alpha$ & $[\alpha,\alpha^3,\alpha^3] s.$                                                       &         $\alpha$          & $[\alpha,\underline{1},\alpha^3,\alpha]ns.$           
	\\	
		3      & $[3,2,4,4]s.$             &                          $\alpha^2$ & $[\alpha^2,\alpha^6,\alpha^6]s.$                                                       &         $\alpha^2$          & $[\underline{\alpha^2},\alpha^3,\alpha^2]ns.$          
	 \\
		4      & $[4,5,3,3]s.$             &                          $\alpha^3$ & $[\alpha^3,\alpha,\alpha]s.$                                                       &       $\alpha^3$          & $[\alpha^3,\underline{1},\alpha,\alpha^5,\alpha^6]ns.$             
	 \\
		5      & $[5,4,\underline{6},4]ns.$             &$\alpha^4$                           &                            $[\alpha^4,\alpha^5,\alpha^5]s.$                            &           $\alpha^4$          & $[\underline{\alpha^4},\alpha^4]ns.$      
	\\
		6      & $[\underline{6},4,4] ns.$  &    $\alpha^5$ & $[\alpha^5,\alpha^4,\alpha^4]s.$                                                      &  $\alpha^5$   & $[\alpha^5,\alpha^3,\underline{\alpha^4}	,\alpha^2,\alpha^7,\alpha^4] ns.$              
		\\
	&              &                          $\alpha^6$ & $[\alpha^6,\alpha^2,\alpha^2]s.$                                                       &        $\alpha^6$          & $[\underline{\alpha^6},\alpha,\alpha^6]ns.$ 
	 \\ 
		&              &                           &  &        $\alpha^7$          & $[\alpha^7,\alpha,\underline{\alpha^4},\alpha^6,\alpha^5,\alpha^4] ns.$    \\
	\hline
	\multicolumn{6}{|c|}{$q=11$}                                                              \\ 
	\hline
	\multicolumn{2}{|c|}{$\underline{d=2}\hspace{3cm}$} & \multicolumn{2}{c|}{$\underline{d=5}\hspace{2cm}$}                                & \multicolumn{2}{c|}{$\underline{d=10}\hspace{3cm}$}   
	\\ 	
		2     & $[2,2]s.$      & 2 & $[2,\underline{1},3,3]ns.$               & 2        & $[2,\underline{10},10]ns.$    
	\\
		3     &  $[\underline{3},6,0,8,6]ns.$     & 3                          &$[3,4,4]s.$  & 3        & $[\underline{3},9,9]ns.$ 
	\\
		4     &    $[\underline{4},1,8,5,10,8]ns.$   & 4                         &$[4,5,5]s.$   & 4        & $[\underline{4},8,8]ns.$           
	\\
		5     & $[\underline{5},9,10,7,0,6,9]ns.$      & 5                         &$[5,6,4,6]s.$            & 5       & $[\underline{5},7,7]ns.$          
						\\
						6     & $[6,8,\underline{3},3]ns$  & 6                 
						&$[6,5,7,5]s.$            & 6       & $[6,6]s.$            
						\\
						7     & $[7,\underline{9},8,2,8]ns.$    & 7                          &$[7,6,6]s.$              & 7       & $[7,\underline{9},8,2,8]ns.$            
						\\
						8     & $[8,\underline{1},4,8]ns.$    & 8                         & $[8,7,7]s.$             & 8       & $[8,\underline{4},4]ns.$             
						\\
						9     & $[\underline{9},6,5,5]ns.$      &  9                          &$[9,\underline{10},8,8]ns.$              &  9      & $[\underline{9},3,3]ns.$
						\\
						10    & $[10,2,\underline{5},4,6,4]ns.$ & 10         
						& $[\underline{10},9,0,10]ns.$                & 10      & $[\underline{10},2,2]ns.$ 
						\\ 
						\hline
						\multicolumn{6}{|c|}{$q=13$}                                                \\
						\hline
						\multicolumn{2}{|c|}{$\underline{d=2}\hspace{3cm}$} & \multicolumn{2}{c|}{$\underline{d=3}\hspace{1.8cm}$}                                & \multicolumn{2}{c|}{$\underline{d=6}\hspace{3cm}$}   \\ 
						
						2  & $[2,2]s.$ &  2 & $[2,10,\underline{1},3,3]ns.$ & 2 & $[2,\underline{10},12,12]ns.$     
						\\
						3  & $[\underline{3},6,7,7]ns.$ & 3 & $[3,4,2,11,\underline{8},8]ns.$ & 3 & $[\underline{3},11,9,11]ns.$     
						\\
						4  & $[\underline{4},12,10,5,8,8]ns.$ &  4 & $[4,3,\underline{5},12,3]ns.$ & 4 & $[\underline{4},10,10]ns.$     
						\\
						5  & $[5,7,5]s.$ & 5 & $[\underline{5},0,5]ns.$ & 5 & $[\underline{5},7,7]ns.$     
						\\
						6 & $[6,\underline{4},10,3,3]ns.$ &  6 & $[6,\underline{1},7,1]ns.$ & 6 & $[6,6]s.$     
						\\
						7 & $[7,\underline{3},2,10,2]ns.$ & 7 & $[7,\underline{12},6,2,2]ns.$ & 7 & $[7,\underline{5},5]ns.$     
						\\
						8 & $[8,\underline{4},8]ns.$ &  8 & $[\underline{8},0,8]ns.$ & 8 & $[\underline{8},4,6,4]ns.$     
						\\
						9 & $[\underline{9},7,1,5,3,0,4,7]ns.$ &    9 & $[9,10,\underline{8},1,10]ns.$ & 9 & $[\underline{9},5,3,5]ns.$     
						\\
						10 & $[10,12,4,6,0,3,12]ns.$ &  10 & $[10,9,11,\underline{1},11]ns.$ & 10 & $[\underline{10},4,4]ns.$     
						\\
						11 & $[11,6,\underline{12},3,11]ns.$ &  11 & $[11,3,\underline{12},10,10]ns.$ & 11 & $[11,\underline{1},3,3]ns.$     
						\\
						12 & $[\underline{12},2,5,0,1,2]ns.$ &  12 & $[\underline{12},11,4,11]ns.$ & 12 & $[\underline{12},2,0,2]ns.$  	
						\\
						\hline	
						\multicolumn{6}{|c|}{$q=13$}  
						\\
						\hline 
						\multicolumn{6}{|c|}{$\underline{d=9 }$} 
						\\
						$ 2 $ & $[2,7,10,\underline{1},3,3]ns.$  &  	$ 3 $ & $[3,4,2,8,11,11]ns.$ &  	$ 4 $ & $[4,3,\underline{5},9,5]ns.$
						\\
						$ 5 $ & $[\underline{5},10,4,4]ns.$ &	$ 6 $ & $[6,11,\underline{1},7,1]ns.$ & $ 7 $ & $[7,2,\underline{12},6,12]ns.$
						\\
						$ 8 $ & $[\underline{8},3,9,9]ns.$ & 	$ 9 $ & $[9,10,\underline{8},4,8]ns.$ & $ 10 $ & $[10,9,11,\underline{5},2,2]ns.$
						\\
						$ 11 $ & $[11,6,3,\underline{12},10,10]ns.$ & $ 12 $ & $[\underline{12},11,7,7]ns.$ &   & 
						\\
						\hline
						\multicolumn{6}{|c|}{$q=16,M(x)=X^4+X+1$}
								\\
									\hline
									\multicolumn{6}{|c|}{$\underline{d=3 }$}
									\\	
									$\alpha$ &  $[\alpha,\underline{\alpha^9},\alpha^{13},\alpha^3,\alpha^3]ns.$   & $\alpha^2$ & $[\alpha^2,{\scriptstyle\underline{\alpha^3},\alpha^{11},\alpha^6,\alpha^6]}ns.$ & $\alpha^3$ & $[\underline{\alpha^3},\alpha,0,\alpha^3]ns.$
									\\
									$\alpha^4$ & $[\alpha^4,\underline{\alpha^6},\alpha^7,\alpha^{12}\alpha^{12}]ns.$ & $\alpha^5$ & $[\alpha^5,\alpha^{10},\alpha^{10}]s.$ &  $\alpha^6$ & $[\underline{\alpha^6},\alpha^2,0,\alpha^6]ns.$ 
									\\
									$\alpha^7$ & $[\alpha^7,\alpha^{10},\underline{\alpha^9},\alpha^2,\alpha^{10}]ns.$  &
									$\alpha^8$ & $[{\scriptstyle\alpha^8,\underline{\alpha^{12}},\alpha^{14}\alpha^9,\alpha^9]}ns.$& $\alpha^9$ & $[\underline{\alpha^9},\alpha^8,0,\alpha^9]ns.$ 	
									\\
									$\alpha^{10}$ & $[\alpha^{10},\alpha^5,\alpha^5]s.$
									& $\alpha^{11}$ & $[\alpha^{11},{\scriptstyle\alpha^5,\underline{\alpha^{12}},\alpha,\alpha^5]}ns.$  & $\alpha^{12}$ & $[\underline{\alpha^{12}},\alpha^4,0,\alpha^{12}]ns.$ 
									\\
									$\alpha^{13}$ & $[\alpha^{13},\alpha^{10},\underline{\alpha^6},\alpha^8,\alpha^{10}]ns.$ & $\alpha^{14}$  & $[{\scriptstyle\alpha^{14},\alpha^5,\underline{\alpha^3},\alpha^4,\alpha^5]}ns.$  &  &   
									\\
									\hline
												\multicolumn{6}{|c|}{$\underline{d=5}$}
	\\
	$\alpha$ & $[\alpha,\alpha^2,\alpha^8,\alpha^8]s.$ & $\alpha^2$ & $[\alpha^2,\alpha^4,\alpha,\alpha]s.$ &  $\alpha^3$ & $[\alpha^3,\alpha^{14},\alpha^{12},\alpha^{14}]s.$ 
	\\
		$\alpha^4$ & $[\alpha^4,\alpha^8,\alpha^2,\alpha^2]s.$ &  $\alpha^5$ & $[\alpha^5,1,\alpha^{10},0,\alpha^5]ns.$ &  $\alpha^6$ & $[\alpha^6,\alpha^{13},\alpha^9,\alpha^{13}]s.$ 
	\\ 
		$\alpha^7$ & $[\alpha^7,\alpha^{13},\alpha^{13}]s.$ & $\alpha^8$ & $[\alpha^8,\alpha,\alpha^4,\alpha^4]s.$  & $\alpha^9$ & $[\alpha^9,\alpha^7,\alpha^6,\alpha^7]s.$
			\\
		$\alpha^{10}$ & $[\underline{\alpha^{10}},1,\alpha^5,0,\alpha^{10}]ns.$  &  $\alpha^{11}$ & $[\alpha^{11},\alpha^{14},\alpha^{14}]s.$ & $\alpha^{12}$ & $[\alpha^{12},\alpha^{11},\alpha^3,\alpha^{11}]s.$  
	\\
$\alpha^{13}$ & $[\alpha^{13},\alpha^7,\alpha^7]s.$ 
		& $\alpha^{14}$ & $[\alpha^{14},\alpha^{11},\alpha^{11}]s.$     &   &
	\\
	\hline
	\multicolumn{6}{|c|}{$q=16,M(x)=X^4+X+1$}                                                    \\ 
	\hline
	\multicolumn{6}{|c|}{$d=9$} 
		\\ 
	 $\alpha $ & $[\alpha,\underline{\alpha^3},\alpha^{13},\alpha^{13}]ns.$  &
		$ \alpha^2 $ & $[\alpha^2,\underline{\alpha^6},{\scriptstyle\alpha^{11},\alpha^{11}]}ns.$& 
		$ \alpha^3 $ & $[\underline{\alpha^3},\alpha^{10},{\scriptstyle\alpha^{14},\alpha^2,0,\alpha^3]}ns.$
		\\
		$ \alpha^4 $ & ${\scriptstyle[\alpha^4,{\underline\alpha^{12}},\alpha^7,\alpha^{11},\alpha^{14},\alpha^{12}]ns.}$& 
		$ \alpha^5 $ & $[\alpha^5,\alpha^{10},\alpha^{10}]s.$& 	$ \alpha^6 $ & $[\underline{\alpha^6},\alpha^5,{\scriptstyle\alpha^{13},\alpha^4,0,\alpha^6]}ns.$
		\\
		$ \alpha^7 $ & $[\alpha^7,\alpha^4,\alpha^{10},\underline{\alpha^9},\alpha^{10}]ns.$& 	$ \alpha^8 $ & $[\alpha^8,\underline{\alpha^9},{\scriptstyle\alpha^{14},}\alpha^{14}]ns.$ &
		$ \alpha^9 $ & $[\underline{\alpha^9},\alpha^5,\alpha^7,\alpha,0,\alpha^9]ns.$
	\\
		$ \alpha^{10} $ & $[\alpha^{10},\alpha^5,\alpha^5]s.$& $ \alpha^{11} $ & ${\scriptstyle[\alpha^{11},\alpha^2,\alpha^5,\underline{\alpha^{12}},\alpha^5]ns.}$& 	$ \alpha^{12} $ & ${\scriptstyle[\underline{\alpha^{12}},\alpha^{10},\alpha^{11},\alpha^8,0,\alpha^{12}]ns}$
		\\
		$ \alpha^{13} $ & $[\alpha^{13}\alpha,{\scriptstyle\alpha^{10},\underline{\alpha^6},\alpha^{10}]}ns.$&$\alpha^{14} $ & ${\scriptstyle[\alpha^{14},\underline{1},\alpha^6,\alpha^{10},\alpha^6]}ns.$  &  &
		\\
		\hline
\end{tabularx}

	\newpage
\begin{tabularx}{11.5cm}{|p{0.25cm}|X|p{0.25cm}|X|p{0.25cm}|X|}
	\hline
	 	\multicolumn{6}{|c|}{$q=17$}                                       
     \\ 
		\hline
		\multicolumn{6}{|c|}{$d=2$} 
		\\
		2 & $[\underline{2},2]ns.$ & 	3 & $[3,6,\underline{16},15,1,15]ns.$ & 	4 & $[\underline{4},12,4]ns.$ 	
		\\	
		5 & $[5,3,\underline{4},11,14,13,4]ns.$ & 	6 & $[6,\underline{13},10,9,7,9]ns.$ & 7 & $[7,\underline{8},6,12,1,11,12]ns.$ 	
	\\
		8 & $[\underline{8},5,0,9,5]ns.$ &  9 & $[\underline{9},4,7,6,10,6]ns.$ &  	10 & $[10,5,\underline{15},11,{\scriptstyle9,3,16,3]}ns.$ 	
		\\
		11 & $[11,\underline{8},2,10,4,5,14,15,10]ns.$ & 	12 & $[12,\underline{13},4,4]ns.$ & 	13 & $[\underline{13},3,13]ns.$ 	
		\\
		14 & $[14,12,11,5,11]s.$ & 	15 & $[\underline{15},6,4,1,{\scriptstyle3,11,4]}ns.$ & 	16 & $[\underline{16},2,6,4,1,2]ns.$ 	
		\\	
		\hline
	\multicolumn{6}{|c|}{$q=19$} 
		\\	\hline
	\multicolumn{6}{|c|}{$d=2$} 
								\\
		2  & $[2,2]s.$ &	3  & $[3,\underline{6},14,3]ns.$  & 
		4  & $[\underline{4},12,7,7]ns.$      
		\\
		5  & $[\underline{5},1,15,11,2,18,15]ns.$ &      
			6 & $[\underline{6},11,1,14,0,13,$ &
			7 & $[\underline{7},4,9,17,16,2,16]ns.$       
	\\
    	&    & $\hspace{2.2cm}11]ns.$ & & &
		\\ 	
		8 & $[8,18,12,3,\underline{1},12]ns.$& 9 & $[\underline{9},15,7,2,14,16,0,$ & 	10 & $[10,14,15,\underline{6},7,1,10]ns.$
		\\
		&  &   &  $\hspace{1.2cm}10,15]ns.$ &   &
	\\
		12 & $[12,18,8,14,13,\underline{5},13]$& 11 & $[\underline{11},15,5,14,14]ns.$ &	13 & $[13,\underline{4},3,15,3]ns.$  
	\\
		14 & $[14,\underline{11},12,16,,14]ns.$ &	15 & $[15,\underline{1},5,10,9,9]ns.$ &	16   & $[\underline{16},12,14,9,8,10,8]ns.$
	\\
	17 & $[\underline{17},6,0,2,6]ns.$ &	18 & $[18,2,{\scriptstyle\underline{5},7,}12,12]ns.$ & & 
		\\  
		\hline 
	\multicolumn{2}{|p{4cm}}{$\underline{d=3 }\hspace{1,7cm}$} & \multicolumn{2}{p{2cm}}{$\underline{d=6}\hspace{3cm}$}                                & \multicolumn{2}{c|}{$\underline{d=9}\hspace{2cm}$}  
		\\
	2  &   $[2,10,14,10]s.$& 2 & $[2,\underline{5},5]ns.$   & 2 & $[2,\underline{1},3,1]ns.$       
	\\
		3 & $[3,\underline{11},4,10,15,15]ns.$ & 3 & $[3,\underline{4},8,17,8]ns.$  & 3 & $[3,2,2]s.$  
	\\
		4 & $[4,\underline{11},5,15,16,15]ns.$   & 4 & $[\underline{4},7,0,15,7]ns$   & 4 & $[4,5,5]s.$     
	\\
		5 & $[5,16,16]s.$      & 5 & $[\underline{5},2,2]ns.$      & 5 & $[5,6,6]s.$       
	\\
   	 6 & $[6,13,\underline{18},5,17,17]$  & 6 & $[\underline{6},5,1,14,1]ns.$  & 6 & $[6,\underline{7},7]ns.$        
	\\
	7 & $[\underline{7},8,6,14,15,0,7]ns.$    & 7 & $[\underline{7},16,0,8,8]ns$    & 7  & $[\underline{7},8,8]ns.$        
		\\
		8 & $[\underline{8},7,9,15,1,9]ns.$      & 8 & $[\underline{8},7,15,3,18,12$      &  8  & $[\underline{8},9,9]ns.$          
		\\
		&   &   & $\hspace{1.5cm} ,15]ns.$ &  & 
		\\
	 9 & $[9,16,\underline{1},19,2,17,1]ns.$ & 9 & $[\underline{9},2,17,17]ns.$  & 9  & $[9,10,\underline{8},8]ns.$    
   \\
	10 & $[10,3,\underline{18},9,17,2,18]ns.$  & 10 & $[10,\underline{1},10]ns$ &  10 & $[10,9,\underline{11},11]ns.$ 
	\\
	11 & $[\underline{11},12,10,4,18,10]ns.$   & 11 & $[\underline{11},9,0,8,9]ns.$    & 11 & $[\underline{11},10,10]ns.$         
		\\
	12 & $[\underline{12},11,13,5,4,0,12]ns.$  & 12 & $[\underline{12},11,8,8]ns.$    & 12 & $[\underline{12},11,11]ns.$        
	\\
	13 & $[13,6,\underline{1},14,5,14]ns.$ & 13 & $[13,\underline{17},13]ns.$ & 13 & $[13,\underline{12},12]ns.$     
	\\
	14 & $[14,3,3]s.$ & 14 & $[14,\underline{12},6,16,12]ns.$& 14 & $[14,13,13]s.$  
	\\
		15 & $[15,\underline{8},14,4,3,4]ns.$  &  15 & $[15,15]s.$  & 15 & $[15,14,14]s.$  
		\\
		16 & $[16,4,4]s.$    &  16 & $[\underline{16},10,14,10]ns.$    & 16 & $[16,\underline{11},4,10,15,15]ns.$  
		\\
		17 & $[17,9,5,9]s.$    &  17 & $[\underline{17},9,13,13]ns.$    & 17 & $[17,\underline{18},16,18]ns.$  
	\\
		18 & $[\underline{18},17,10,11,0,18]ns.$    &  18 & $[\underline{18},2,8,2]ns.$    & 18 & $[\underline{18},17,0,18]ns.$     
	\\
	\hline
	\multicolumn{6}{|c|}{$q=23$} 
	\\
	\hline
	\multicolumn{6}{|c|}{$\underline{d=2}\hspace{1cm}$}   
	\\
	2 & $[\underline{2},2]ns.$ & 3 &  $[\underline{3},6,10,5,22,21,1,$   & 
	4 & $[\underline{4},12,2,0,19,12]ns.$ 
	\\
		&  &   & $\hspace{1.6cm}21]ns.$  &   &
    \\
		5 & $[5,20,4,11,\underline{1},19,11]ns.$   & 	6 & $[\underline{6},7,20,3,3]ns.$ &   7 & $[7,19,\underline{9},5,18,18]ns.$    
	\\ 
	8 & $[\underline{8},10,0,15,10]ns.$ & 9 & $[\underline{9},3,0,14,3]ns.$ & 10 & $[10,{\scriptstyle21,17,\underline{3},22,14,2,17]}ns.$ 
    \\ 	
	11 & $[11,\underline{18},14,1,13,20,21,16,$ &	12 & $[\underline{12},17,1,12]ns.$ &  	13 & $[\underline{13},18,12,15,13]ns.$   	
	\\ 
	14 & $[14,21,\underline{13},17,22,10,17]ns.$  &  	15 & $[15,3,17,21,\underline{12},14,$ & 16 & $[\underline{16},10,15,{\scriptstyle2,11,13,}15]ns$
	\\ 
	&  &   & $\hspace{1.2cm}20,17]ns.$  &   & 
    \\
	17 & $[17,19,22,7,\underline{9},18,8,1,7]ns.$   &   18 & $[\underline{18},7,8,0,5,7]ns.$  & 19 & $[19,20,\underline{13},12,10,12]ns.$  
	\\
		20 & $[20,\underline{12},9,15,21,7,6,16,6]ns$   &   21 & $[21,\underline{6},15,20,11,8,$     & 22 & $[22,\underline{2},5,3,9,13,9]ns.$
	\\
	&   &    &  $\hspace{1.6cm}20]ns.$ &   &
	\\
		\hline
	\multicolumn{6}{|c|}{$q=25, M\left(X\right)=X^{2}+4X+2 $} 
   \\
	\hline
	\multicolumn{2}{|c}{$\underline{d=2 }\hspace{3.2cm}$} &
	\multicolumn{2}{c}{$\underline{d=3 }\hspace{1.6cm}$} & \multicolumn{2}{c|}{$\underline{d=6}\hspace{2.7cm}$}  
	\\		
	$\alpha $    & $[\alpha,\underline{\alpha^{18}},\alpha^{10},\alpha^{21},\alpha^3,\alpha^{14},]$ &	$\alpha$    & $[\alpha,\underline{\alpha^{18}},\alpha^{15},\alpha^{20},$  &  $\alpha$    & $[\alpha,\underline{\alpha^{14}},\alpha^5,\alpha^{14}  ]ns.$ 
    \\
	& $\hspace{2cm}\alpha^{15}, \alpha^{14}]ns.$ &   & $\hspace{0.9cm}\alpha^{17},\alpha^{18}]ns$   & &
	\\
	$\alpha^2$ & $[\underline{\alpha^2},\alpha^{17},\alpha^9,\alpha,0,{\scriptstyle\alpha^{14},\alpha^{17}]}ns.$&	$\alpha^2$    & $[\alpha^2,\alpha,\underline{1},\alpha^{17},1 ]ns.$ & 	$\alpha^2$    & $[\underline{\alpha^2},\alpha^5,\alpha^{10}\alpha^{5} ]ns.$  
	\\
	$\alpha^3$ & $[\alpha^3,\alpha^{17},\underline{1},\alpha^2,\alpha^{20},\alpha^{13},\alpha^7,$&	$\alpha^3$  & $[\underline{\alpha^3},\alpha^{21,0,\alpha^3} ]ns.$ & 	$\alpha^3$ & $[\underline{\alpha^3},\alpha,\alpha^{17},\alpha^{17}]ns.$ 	  
	\\
	$\hspace{2.8cm}  \alpha^{12},\alpha^2]ns.$&   &   &   &   & 
	\\
	$\alpha^4$  & $[\underline{\alpha^4},\alpha^{12},\alpha^{20},\alpha^{22},\alpha^{15},\alpha^{19},$&	$\alpha^4$    
	& $[\alpha^4,\alpha^8,\alpha^{23},\alpha^5, $ & 	$\alpha^4$    & $[\underline{\alpha^4},\alpha^{20},\alpha^{20}]ns.$  
	\\
	& $\hspace{2.3cm} \alpha^7, \alpha^7]ns.$ &   &$\hspace{1.3cm}\alpha^{20}  \alpha^8]s.$ &  &
	\\ 
	$\alpha^5$  & $[\alpha^5,\underline{\alpha^{18}},\alpha^2,\alpha^9,\alpha^{15},\alpha^{22},$ &	$\alpha^5$    & $[\alpha^5,\underline{\alpha^{18}},\alpha^3,\alpha^4,$ & 	$\alpha^5$  & $[\alpha^5,\underline~{\alpha^{22}},\alpha^2,\alpha^2]ns.$
	\\
	&  $\hspace{2.6cm}\alpha^9, ]ns.$  &   & $ \hspace{0.9cm} \alpha^{13},\alpha^{18} ]ns.$  &   & 
	\\
	$\alpha^6$ & $[\underline{\alpha^6},\alpha^6]ns.$ &
	$\alpha^6$    & $[\underline{\alpha^6},0,\alpha^6 ]ns.$   &     
	$\alpha^6$    & $[\underline{\alpha^6},\alpha^6 ]ns.$ 		
	\\ 
	$\alpha^7$  & $[\alpha^7,\underline{\alpha^4},\alpha^{24},\alpha^9,\alpha^{16},\alpha^{24}]ns.$ &
	$\alpha^7$    & ${\scriptstyle [\alpha^7,\alpha^{10},\alpha^4,{\underline\alpha^{21}},\alpha^{11}},$&	
	$\alpha^7$ & $[\alpha^7,\underline{\alpha^{16}},\alpha^9,\alpha^{11}\alpha^{16}]ns.$
	\\ 
	& $\hspace{2cm}$   &   & $\hspace{0.7cm}{\scriptstyle \alpha^{24},\alpha^{8}, \alpha^{8} ]n}s.$   & &
	\\
	$\alpha^8$  & $[\underline{\alpha^8},\alpha^{15},\alpha^9,\alpha^{13},\alpha^{14},\alpha^{24}]$ &
	$ \alpha^8$  & $[{\scriptstyle\alpha^8,\alpha^4,\alpha^7,\alpha^{13},\alpha^{16},}$ &  
	$\alpha^8$    & $[\underline{\alpha^8},\alpha^{19},\alpha^{11},\alpha^{11} ]ns.$
	\\
	&  &  &$\hspace{1.7cm} {\scriptstyle\alpha^4} ]s.$ &   &
	\\	
	$\alpha^9$  & $[\alpha^9,\underline{\alpha^4},\alpha^{13},\alpha^{11},\alpha^{19},\alpha^{22}] $&
	$\alpha^9$    & ${\scriptstyle [\underline{\alpha^9},\alpha^{21},\alpha^3,\alpha^{21}]}ns.$ & $\alpha^9$    & $[\underline{\alpha^9},\alpha^8,\alpha^7,\alpha^4,\alpha^7 ]ns.$  
	\\ 
	&   &      &   &   &
	\\
	$ \alpha^{10}$  & $[\underline{\alpha^{10}},\alpha^{13},\alpha^8,\alpha^{10}]ns.$&
	$\alpha^{10}$ & $  {\scriptstyle [ \alpha^{10},\alpha^5,\underline{1},\alpha^{13},\alpha^{24}]ns.}$ & $\alpha^{10}$ & $[\underline{\alpha^{10}},\alpha,\alpha^2,\alpha]ns.$   
\\
	\hline
	\end{tabularx}

	\newpage
	\begin{tabularx}{11.5cm}{|p{0.25cm}|X|p{0.25cm}|X|p{0.25cm}|X|}
		\hline
	\multicolumn{6}{|c|}{} 
	\\	
	$\alpha^{11}$ & ${\scriptstyle[\alpha^{11},\underline{\alpha^{20}},\alpha^{24},\alpha^8,\alpha^{24}]}$&
	$\alpha^{11}$  & $[\alpha^{11},\alpha^2,\alpha^{20},\underline{\alpha^9},\alpha^7, $ & 	$\alpha^{11}$    & $[\alpha^{11},\underline{\alpha^8},\alpha^{21},\alpha^7,\alpha^8 ]ns.$
	\\
	&  &  & $\hspace{0.9cm} \alpha^{24}, \alpha^{16},\alpha^{16} ]ns.$&  &
	\\
	$\alpha^{12}$ & $[\underline{\alpha^{12}},\alpha^6,0,\alpha^{24},\alpha^6]ns$ &
	$\alpha^{12}$    & $[\underline{\alpha^{12}},\alpha^{18},\alpha^{24},\alpha^{12} ] ns.$  &  	$\alpha^{12}$    & $[\underline{\alpha^{12}},\alpha^6,0,\alpha^{24},\alpha^6 ]ns.$
	\\ 
	$\alpha^{13}$  & $[\alpha^{13},\alpha^{23},\underline{\alpha^8},{\scriptstyle\alpha^3},{\scriptstyle \alpha^{15}]}ns$ &
	$\alpha^{13}$ & $[\alpha^{13},\underline{\alpha^6},\alpha^3,\alpha^8,\alpha^5,$ &  $\alpha^{13}$    & $[\alpha^{13},\underline{\alpha^{15}},{\scriptstyle \alpha^2},{\scriptstyle \alpha^{17}},{\scriptstyle\alpha^{15}  ]ns.}$
	\\
	$\alpha^{15}$ & $[\alpha^{15},\alpha^{13},\underline{1},\alpha^{10},\alpha^4 ,  $ &  $\alpha^{15}$   & $[\underline{\alpha^{15}},\alpha^9,0,\alpha^{15}]ns.$  & 	$\alpha^{15}$    & $[\underline{\alpha^{15}},\alpha^5,\alpha^{13},\alpha^{13}  ]ns.$ 
	\\ 
	&  $\hspace{2cm}\alpha^{17}] ns.$ &  &  &  &
	\\
	$\alpha^{16}$  & $[\underline{\alpha^{16}},\alpha^3,\alpha^{21},\alpha^7,\alpha^{17},$& $\alpha^{16}$    & $[\alpha^{16},\alpha^{20},\alpha^{11},\alpha^{17},\alpha^8,$ &  $\alpha^{16}$    & $[\underline{\alpha^{16}},\alpha^{23},\alpha^{7},\alpha^{7}  ]ns.$
	\\ 
	& $\hspace{2.2cm}\alpha^{22}]$ &   &  $\hspace{2.2cm}\alpha^{20}]s.$ &   &
	\\
	$\alpha^{17}$  & ${\scriptstyle[\alpha^{17},\alpha^{19},\underline{\alpha^{16}},\alpha^{15},\alpha^3,}$ &
	$\alpha^{17}$    & $[\alpha^{17},\underline{\alpha^6},\alpha^{15},\alpha^{16},\alpha,$ &  $\alpha^{17}$    & $[\alpha^{17},\underline{\alpha^3},\alpha^{10},{\scriptstyle\alpha^{13},}\alpha^3 ]ns.$
	\\
	& $\hspace{2.2cm}\alpha^3]ns.$ &   & $\hspace{2.2cm}\alpha^6 ]ns.$  &  &
	\\
	$\underline{\alpha^{18}}$ & $[\alpha^{18},1,\alpha^{18}]ns.$  &	$\alpha^{18}$    & $[ \underline{\alpha^{18}},0,\alpha^{18} ]ns.$ & 	$\alpha^{18}$    & $[ \underline{\alpha^{18}},\alpha^{24},\alpha^{18} ]ns.$      
	\\ 
	$\alpha^{19}$  & $[\alpha^{19},\alpha^{15},\underline{\alpha^4},\alpha^5,\alpha^{17},$& 	$\alpha^{19}$    & $[\alpha^{19},{\scriptstyle\alpha^{22};\alpha^{16};\alpha^{19}};\alpha^{9}]$ns. & 	$\alpha^{19}$    & $[\alpha^{19},\alpha^{23},\alpha^{23}  ]s.$ 
	\\ 
	& $\hspace{2cm}\alpha^{17}]ns.$ &   &   &   &
	\\
	$\alpha^{20}$  & $\underline{[\alpha^{20}},{\scriptstyle\alpha^{12},\alpha^4,\alpha^{14},\alpha^3,\alpha^{23}]}$ & $\alpha^{20}$    & $[\alpha^{20}\alpha^{16},{\scriptstyle\alpha^{19},\alpha,\alpha^4,\alpha^{16} ]s.}$  & 	$\alpha^{20}$    & $[\underline{\alpha^{20}}\alpha^{4},\alpha^{4}]ns.$  
	\\
	$\alpha^{21}$    & $[\underline{\alpha^{21}},{\scriptstyle\alpha^{16},\alpha^{11}, \alpha^{20},\alpha^{11}  ]ns.}$  & $\alpha^{21}$    & $[\underline{\alpha^{21}},{\scriptstyle\alpha^{3},\alpha^{15},\alpha^9,0,\alpha^  {21}]ns.}$   &	$\alpha^{21}$    & $[\alpha^{21},\alpha^{23},\underline{\alpha^{10}},\alpha^{10}  ]ns.$
		\\    
	$\alpha^{22}$ & $[\underline{\alpha^{22}},\alpha^{23},0,\alpha^{10},\alpha^{23}ns.]$ &
		$\alpha^{22}$ & $[\alpha^{22},\alpha^{17},\underline{\alpha^{12}},\alpha,\alpha^{12}]ns.$ & 	$\alpha^{22}$    & $[\underline{\alpha^{22}},\alpha^{3},\alpha^{14},\alpha^3  ]ns.$
	\\ 
	$\alpha^{23}$  & $[\alpha^{23},\alpha^3,\underline{\alpha^{20}},\alpha,\alpha^{13},$&
	$\alpha^{23}$    & $[\alpha^{23},\alpha^{14},\alpha^8,\underline{\alpha^{21}},\alpha^{19},$ & $\alpha^{23}$    & $[\alpha^{23},\alpha^{19},\alpha^{19}]s.$ 
	\\
	& $\hspace{2cm}\alpha^{13}]ns.$  &  & $\hspace{1.5cm}\alpha^{12} ,\alpha^{4},\alpha^{4}]$ &  &
	\\
	\hline
	\multicolumn{6}{|c|}{$q=25,M(X)=X^2+4X+2$}
	\\
	\hline
	\multicolumn{6}{|c|}{$\underline{d=9}$} 
	\\
	$\alpha$    & $[\alpha,\alpha^{5}, \alpha ]s.$  & $\alpha^9$    & $[\underline{\alpha^9},\alpha^{15},\alpha^3,\alpha^{23},\alpha^3 ]ns.$  &  	$\alpha^{17}$    & $[\alpha^{17},\alpha^7,\alpha^{8},\alpha,\alpha^7  ]s.$
	\\
	$\alpha^2$   & $[\alpha^2,\alpha,\alpha^{10},\alpha  ]s.$& 	$\alpha^{10}$    & $[\alpha^{10},\alpha^{14},\alpha^5,\alpha^2,\alpha^{14} ]s.$   & $\alpha^{18}$    & ${\scriptstyle[ \underline{\alpha^{18}},\alpha^{24},\alpha^{12},\alpha^6,0,\alpha^{18} ]ns.}$
	\\
	$\alpha^3$    & $[\underline{\alpha^3},\alpha^{9},{\scriptstyle\alpha^{21}, \alpha^{15}, 0,} \alpha^3 ]ns.$ &  $\alpha^{11}$  & $[\alpha^{11},\alpha^7,\alpha^{10},\alpha^{19},\alpha^7  ]s.$  & 	$\alpha^{19}$    & $[\alpha^{19},{\scriptstyle\alpha^{23},\alpha^{14}},\alpha^{11},\alpha^ {23} ]s.$
	\\
	$\alpha^4$    & $[\alpha^4 ,\alpha^8,\alpha^{23},\alpha^{20},\alpha^8 ]s.$  &  $\alpha^{12}$    & ${\scriptstyle[\underline{\alpha^{12}},\alpha^{18},\alpha^6, \alpha^{24},0,\alpha^{12} ]ns.}$  & 	$\alpha^{20}$    & $[\alpha^{20} ,\alpha^{16},\alpha^{19},\alpha^4,\alpha^ {16}]s.$  
	\\
	$\alpha^5$    & $[\alpha^5,\alpha,\alpha^4,\alpha^{13},\alpha  ]s.$ &  	$\alpha^{13}$    & $[\alpha^{13},\alpha{17},\alpha^8,\alpha^5,\alpha^{17}]s.$  & $\alpha^{21}$    & $[\underline{\alpha^{21}},{\scriptstyle\alpha^{3},\alpha^{15},\alpha^9,0,\alpha^  {21}]}ns.$
	\\
	$\alpha^6$   & $[\alpha^6,\alpha^{12},\alpha^{24},\alpha^{18},0,\alpha^6 ]$&   $\alpha^{14}$    & $[\alpha^{14},{\scriptstyle\alpha^{10},\alpha^{13},\underline{\alpha^{12}},\alpha^  {10}]}ns.$   & $\alpha^{22}$    & $[\alpha^{22},\alpha^2,\alpha^{17},\alpha^{14},\alpha^  2]s.$ 
	\\
	$\alpha^7$    & $[\alpha^7,\alpha^{11},\alpha^2,\alpha^{23} ,\alpha^{11} ]s.$ & $\alpha^{15}$    & $[\underline{\alpha^{15}},{\scriptstyle\alpha^{21},\alpha^{9},\alpha^{3},0,\alpha^{15} ]}ns.$  &  & 
	\\
	$\alpha^8$  & $[\alpha^8 ,\alpha^4,\alpha^7,\alpha^{16},\alpha^4 ]s.$  &  $\alpha^{16}$    & $[\alpha^{16} ,\alpha^{20},\alpha^{11},\alpha^8,\alpha^{20} ]s.$  &  & 
	\\	
	\hline
	\multicolumn{6}{|c|}{$q=27$, {}$M\left(X\right)=X^{3}+X+1$}
	\\
	\hline
	\multicolumn{6}{c|}{$\underline{d=2}$} 
	\\
	$\alpha$    & $[\alpha,\underline{\alpha^4 },\alpha^{13},\alpha^7,\alpha ]ns.$ &  $\alpha^{10}$    & $[\underline{\alpha^{10}}, \alpha^{21},\alpha^{20},\alpha^{17},\alpha^7,$  & 	$\alpha^{19}$    & ${\scriptstyle [\alpha^{19},\alpha^{17},\alpha,\underline{\alpha^{20}}},\alpha^{21},\alpha^{12},$ 
	\\
	&  & & $\alpha^{17} ]ns.$ &  & ${\scriptstyle\alpha^{13},\alpha^{11},}\alpha^2,\alpha^{25},\alpha^{13} ]ns.$
	\\
	$\alpha^2$  & $[\underline{\alpha^2},\alpha^{14},0,\alpha^{15},\alpha^{14}  ]ns.$ & $\alpha^{11}$ & $[\alpha^{11},\alpha^{17},\underline{\alpha^4},\alpha^4  ]ns.$  &  $\alpha^{20}$    & $[\underline{\alpha^{20}},\alpha^{11},{\scriptstyle\alpha^6,}\alpha^{24},\alpha^6 ]ns.$
	\\
	$\alpha^3$ & $[\alpha^3,\underline{\alpha^{12}},\alpha^5,\alpha^{21},\alpha^3  ]ns.$ &  $\alpha^{12}$    & $[\underline{\alpha^{12}},\alpha^7,\alpha^7 ]ns.$  &   $\alpha^{21}$ & $[\alpha^{21},\alpha^{23},\underline{\alpha^{10}},\alpha^{10}  ]ns.$
	\\
	$\alpha^4$ & $[\underline{\alpha^4} ,\alpha^{11},\alpha^8,\alpha^{25},\alpha^{21},$ &  $\alpha^{13}$    & $[\alpha^{13},\alpha^{13} ]s.$ & 	$\alpha^{22}$    & $[\underline{\alpha^{22}},\alpha^{12},\alpha^{8},\alpha^{13},\alpha^3,$
	\\
	& $\hspace{2cm}\alpha^{25}]ns.$  &   &   &  & $\hspace{1.7cm}\alpha^7,\alpha^3 ]ns.$ 
	\\
	$\alpha^5$    & $[\alpha^5,\alpha^{25},\alpha^3,{\scriptstyle\underline{\alpha^8},\alpha^{11},\alpha^{10},}$ &  $\alpha^{14}$    & $[\underline{\alpha^{14}},\alpha^{10},\alpha^{24},\alpha^9,\alpha^{21},$ &  	$\alpha^{23}$    & $[\alpha^{23},\underline{\alpha^{16}},\alpha^{24},\alpha^{12},1,$
	\\
	&$ \hspace{0.4cm} \alpha^4,\alpha^{14},\alpha^{24},\alpha^{10} ]ns.$  &    & $ \hspace{1.7cm} 1, \alpha^9, ]ns.$   &    & $\hspace{1.1cm}\alpha^6, \alpha^5,\alpha^{23}  ]ns.$	
	\\
	$\alpha^6$    & $[\underline{\alpha^6},\alpha^{16},0,\alpha^{19},\alpha^{16}]ns.$ &  $\alpha^{15}$    & $[\alpha^{15},{\scriptstyle\alpha^{23},\alpha^9,\underline{\alpha^{24}}},\alpha^7,\alpha^4, $ &  $\alpha^{24}$    & $[\underline{\alpha^{24}},\alpha^{21},\alpha^2,\alpha^8,\alpha^2  ]ns.$ 
	\\
	&  &   & ${\scriptstyle\alpha^{13},\alpha^{21},\alpha^{18},}\alpha^{17},\alpha^{13}  ]ns.$ &  &
	\\
	$\alpha^7$    & $[\alpha^7,\alpha^{25},\alpha^{12},\underline{\alpha^{12}}  ]ns.$  & 		$\alpha^{16}$    & $[\underline{\alpha^{16}},\alpha^4,\alpha^{20},\alpha^{13},\alpha,$  & $\alpha^{25}$    & $[\alpha^{25},\underline{\alpha^{14}},\alpha^8,\alpha^4,1,\alpha^2, $ 
	\\
	&   &   & $\hspace{1.4cm}\alpha^{11},1,\alpha ]ns.$  &   &$\hspace{1.4cm}\alpha^{19},\alpha^{25}]ns.$
	\\
	$\alpha^8$    & $[\underline{\alpha^8},\alpha^7,\alpha^{18},{\scriptstyle\alpha^{20},}\alpha^{18} ]ns.$ & $\alpha^{17}$    & $[\alpha^{17},\underline{\alpha^{22}},\alpha^{20},\alpha^{10},1,$  &   &
	\\
	&  &  & $\hspace{0.8cm}\alpha^{18},\alpha^{15},\alpha^{17}  ]ns.$ &  &
	\\
	$\alpha^9$   & $[\alpha^9,\underline{\alpha^{10}},\alpha^{15},{\scriptstyle\alpha^{11},}\alpha^{9}  ]ns.$  &  
	$\alpha^{18}$    & $[\underline{\alpha^{18}},\alpha^{22},0,\alpha^5,\alpha^{12} ]ns.$  &  & 
	\\
	\hline	
	\end{tabularx}

\end{document}